\documentclass[12pt,twoside,reqno]{amsart}

\usepackage{amssymb,amsmath,amstext,amsthm,amsfonts,amscd}

\usepackage{leftidx}

\usepackage[ansinew]{inputenc} 
\usepackage{graphicx}
\usepackage[mathscr]{eucal}
\usepackage{hyperref}
\usepackage{color}
 \usepackage{arydshln}

\newcommand{\C}{\mathbb{C}}
\newcommand{\N}{\mathbb{N}}

\newcommand{\Mat}{{\rm Mat}}

\newcommand{\Bscr}{\mathscr{B}}
\newcommand{\Sc}{{S^c}}

\newcommand{\sabs}[1]{\left| #1 \right|} 
\newcommand{\abs}[1]{\bigl| #1 \bigr|} 
\newcommand{\norm}[1]{\lVert#1\rVert}

\newcommand{\normtwo}[1]{
{\left\vert\kern-0.25ex\left\vert\kern-0.25ex\left\vert #1 
    \right\vert\kern-0.25ex\right\vert\kern-0.25ex\right\vert} }

\newcommand{\FF}{\mathscr{F}}

\newcommand{\dist}{{\rm dist}}
\newcommand{\supp}{{\rm supp}}
\newcommand{\spec}{{\rm sp}}

\newcommand{\Aop}{\mathscr{A}}
\newcommand{\Qop}{\mathscr{Q}}
\newcommand{\Rop}{\mathscr{R}}
\newcommand{\Dop}{\mathscr{D}}
\newcommand{\Lops}{\mathcal{L}}

\newcommand{\ind}[1]{1_{{#1}}}

\newcommand{\D}{D}

\newcounter{main}

\numberwithin{equation}{section}

\newtheorem{theorem}{Theorem}[section]
\newtheorem{proposition}[theorem]{Proposition}
\newtheorem{lemma}[theorem]{Lemma}

\newtheorem{remark}{Remark}[section]

\newtheorem{definition}{Definition}[section]
\newtheorem{maintheorem}{Theorem}

\newcommand{\blanksquare}{\,\,\,$\sqcup\!\!\!\!\sqcap$}

{\par}

\title[Isospectral Reduction in Infinite Graphs]{\textbf{Isospectral Reduction in Infinite Graphs}}

\author[P. Duarte]{Pedro Duarte}
\address{CMAF, Departamento de Matem\'{a}tica,
  Faculdade de Ci\^{e}ncias da Universidade de Lisboa, Campo Grande, 1749-016 Lisboa, Portugal.}
\email{pedromiguel.duarte@gmail.com}

\author[M. J. Torres]{Maria Joana Torres}
\address{CMAT and Departamento de Matem\'atica e Aplica\c{c}\~{o}es, 
Universidade do Minho, 
Campus de Gualtar, 
4700-057 Braga, Portugal}
\email{jtorres@math.uminho.pt}

\begin{document}

\maketitle

\begin{abstract}
L. A. Bunimovich and B. Z. Webb developed a theory for 
transforming a finite weighted graph while preserving its  spectrum, referred as isospectral reduction theory.
In this work we extend this theory to a class of operators on Banach spaces that include Markov type operators. 
We apply this theory to infinite countable weighted graphs admitting a 
finite structural set to calculate the stationary measures of a family of countable Markov chains. 
\end{abstract}

\bigskip


\smallskip


\bigskip


\section{Introduction}
\label{intro}

L.A. Bunimovich and  B.Z. Webb developed a theory for isospectral 
graph reduction in finite dimensional graphs  
(see~\cite{BW,BW3,BW2}). This procedure maintains the spectrum of the graph's adjacency matrix up to a set of eigenvalues known beforehand from its graph structure. 
More precisely, the authors introduce a concept of transformation of a graph (either by reduction or expansion)
that can be used to simplify the structure of a graph 
while preserving the eigenvalues of the graph's adjacency matrix.
In order to not contradict the fundamental theorem of algebra,  isospectral graph transformations preserve the spectrum of the graph (in particular the number of eigenvalues) by permitting edges to be weighted by functions of a spectral parameter $\lambda$ (see ~\cite[Theorem 3.5.]{BW}).
Thus such transformations allow one to modify the topology of a network (changing the interactions, reducing or increasing the number of nodes),
while maintaining properties related to the network's dynamics.

More recently, in~\cite{DT}, we have proven that isospectral graph reductions also preserve the eigenvectors associated 
with the eigenvalues of the graph's weighted adjacency matrix. We explain how the isospectral reduction procedure can be used to efficiently  update the eigenvector of a large sparse matrix when only a small number of its  entries is modified. As an
application we propose an updating algorithm for the maximal eigenvector of the Markov matrix associated to 
 a large sparse dynamical network.

Because our spectral approach to  isospectral graph reduction theory
is based on {\em eigenvectors}, instead of {\em eigenvalues}, it was a natural question to ask about possible generalizations of this theory to infinite dimensions.

We believe there are many possible such extensions to infinite dimensional models. In this work we develop a couple of abstract settings where such generalizations hold.

The theory applies to a class of bounded linear operators acting on
spaces of $L^1$-integrable functions. The operators considered are written as a sum  of a diagonal plus a Markov operator.
A key concept in Bunimovich-Webb's isospectral theory is that of  a {\em structural set}. In this work we give three  different concepts of structural sets  (see Definitions~\ref{structural:A},~\ref{structural:B} and~\ref{structural:AB}) and for each of them  prove  a corresponding isospectral theorem (see Theorems~\ref{main:A},~\ref{main:B} and~\ref{main:AB}).

The theory developed can be used to handle  a  wide class of examples. 
An application of Theorem~\ref{main:AB} is given to weighted countably infinite graphs with a finite structural set (see Theorem~\ref{main:AB:countable}). 
We also propose a numerical algorithm
to approximate the eigenfunctions of such weighted graphs.
We conclude the manuscript with a concrete application of the theory
to calculate the stationary measures of 
a family of
infinite Markov chains.

The paper is organized as follows:

In Section~\ref{reduction:finite} we describe the isospectral graph reduction theory and the reduction statements for finite graphs.

In Section~\ref{infinitemodel} we generalize the  isospectral graph reduction theory to infinite dimensional models.

In Section~\ref{countablegraphs} we apply the infinite dimension isospectral reduction theory, developed in Section~\ref{infinitemodel}, to countably infinite graphs with a 
finite structural set. 
We also propose a numerical algorithm to approximate the eigenfunctions of such graphs.

In Section~\ref{Markov} we present an example where the theory is applied to give a closed formula for the 
stationary probability measures of a family of infinite 
Markov chains.

\bigskip


\section{Finite graphs}
\label{reduction:finite}

In this section we describe the isospectral graph reduction theory and the reduction statements for finite graphs.

\bigskip

\begin{definition} 
\label{graph:finite} \,
A  {\em finite weighted graph} is a pair
$G=(V,w)$ where 
$V$
is a finite set and 
$w \colon V\times V\to\C$ is any function, called the weight function of $G$. We denote by 
$\Aop=\Aop_w \colon \C^V\to \C^V$ 
the operator defined by the weighted adjacency matrix $(w(i,j))_{i,j \in V}$.
\end{definition}

A {\em path} $\gamma=(i_0,\ldots, i_p)$ in the graph $G=(V,w)$ is an ordered sequence of  vertices 
$i_0,\ldots, i_p \in V$ such that $w(i_\ell,i_{\ell+1}) \neq 0$ for $0 \leq \ell \leq p-1$. The integer $p$ is called the length of   $\gamma$.
If   the vertices $i_0,\ldots, i_{p-1} $ are all distinct
the path $\gamma$ is called {\em simple}.
If $i_0=i_p$ then $\gamma$ is called a {\em closed path}.
A closed path of length $1$ is called a {\em loop}.
Finally, we call {\em cycle} any simple closed path.

If $S \subseteq V$ we will
 write $\Sc=V \setminus S$. 

\bigskip

\begin{definition} [Structural set]
\label{finite:structural} \,
Let $G=(V,w)$. 
A nonempty vertex set  $S \subseteq V$ is a {\em structural set} for $G$ if each cycle of $G$, that is not a loop, contains a vertex in $S$.
\end{definition}

\bigskip

Given a structural set $S$, we call {\em branch of}  $(G,S)$ 
to 
any 
simple path $\beta=(i_0,i_1,\ldots, i_{p-1}, i_p)$  
 such that  $i_1,\ldots, i_{p-1}\in \Sc$ and  $i_0, i_p\in V$.
 We denote by $\Bscr=\Bscr_{G,S}$ the set of all branches of $(G,S)$.
 Given vertices $i,j\in V$, we denote by
$\Bscr_{i j}$  the set of all branches in $\Bscr$ that start in $i$ and end in $j$. 
Define $\Sigma:=\{w(i,i):\, i \in \Sc\}$ and let $\lambda \in \C \setminus \Sigma$.
For each branch $\beta=(i_0,i_1,\ldots, i_p)$ we define
the {\em $\lambda$-weight of $\beta$} as follows:
\begin{equation}\label{weight}
w(\beta,\lambda):= w(i_0,i_1)\,\prod_{\ell=1}^{p-1} \frac{w(i_{\ell},i_{\ell+1})}{\lambda-w(i_{\ell},i_\ell)} \;.
\end{equation}
\bigskip
Given $i,j \in V$ set
\begin{equation}\label{reduced:matrix}
 R_{S,\lambda}(i,j):= \sum_{\beta\in \Bscr_{i j}} w(\beta,\lambda)\;.
\end{equation}
\bigskip
The {\em reduced operator} 
$\Rop_S(\lambda) \colon \C^S \to \C^S$ is the operator with matrix
$(R_{S,\lambda}(i,j))_{i,j \in S}$.

In~\cite{BW} the reduced operator $\Rop_S$ is viewed as a matrix
indexed in $S\times S$ with values in the field $\mathbb{W}[\lambda]$ of all rational functions $f(\lambda)=\frac{p(\lambda)}{q(\lambda)}$, where $p(\lambda)$ and $q(\lambda)$ are polynomials. In their treatment Bunimovich and Webb consider, more generally, weighted adjacency matrices $\Aop$ with  values in the field $\mathbb{W}[\lambda]$ instead of $\C$, so that the reduced matrix $\Rop_S$ lives in the same space of  $\mathbb{W}[\lambda]$-valued matrices.
Given  a matrix $\Aop(\lambda)\in \mathbb{W}[\lambda]^{V\times V}$ its spectrum is defined in~\cite[Definition 3.1]{BW}
by $\spec(\Aop(\lambda)) = P\setminus Q$ 
where $P=\{\lambda\in\C\colon p(\lambda)=0\}$,
$Q=\{\lambda\in\C\colon q(\lambda)=0\}$ and $\det(\Aop(\lambda)-\lambda\,I)= p(\lambda)/q(\lambda)$.
In the context of the previous definitions, 
starting with a complex valued matrix $\Aop\in\C^{V\times V}$,
by~\cite[Corollary 3]{BW}
 the spectrum of the reduced operator $\Rop_S(\lambda)$
 matches the following definition, which is more suitable for our infinite dimensional isospectral reduction.

\begin{definition}
We define the {\em spectrum} of the family of operators
$\Rop_S(\lambda)$, denoted by
$\spec(\Rop_S)$,   to be 
$$\spec(\Rop_S):=\left\{ \lambda\in \C\setminus \Sigma  \colon
\det(\Rop_S(\lambda)-\lambda\,I) = 0  \right\} . $$
\end{definition}

\bigskip

A simplified\footnote{This statement corresponds to~\cite[Corollary 3]{BW} where the  adjacency matrix has complex entries.} version of  Bunimovich-Webb isospectral reduction theorem 
(see~\cite[Theorem 3.5.]{BW})
can be stated as follows:

\begin{maintheorem}[Bunimovich-Webb]
\label{BW thm} \,
Given a structural set $S$ for a graph $G=(V,w)$,
$$ \spec (\Aop)\setminus \Sigma = \spec(\Rop_S) . $$
\end{maintheorem}

\bigskip

We have stated our reduction results in~\cite[Theorem 1, Proposition 2.1]{DT}  in terms of restriction and extension of eigenvectors.
 The following theorem states that 
isospectral graph reduction  preserves the eigenvectors associated 
with the eigenvalues of the graph's weighted adjacency matrix. 

\bigskip

\begin{maintheorem}[{\cite[Theorem 1]{DT}}] \,
\label{finite:reduction} 
Given a graph $G=(V, w)$, let
$\lambda_0 \in \C \setminus \Sigma$ be an eigenvalue of 
$\Aop=\Aop_w \colon \C^V\to \C^V$ 
and 
$u\in \C^V$
be 
a corresponding eigenvector, 
$\Aop \,u = \lambda_0 u$. Assume that $S$ is a structural set for $G$. Then $\lambda_0$ is also an eigenvalue of $\Rop_S(\lambda_0)$ and $\Rop_S(\lambda_0)\, u_S = \lambda_0 u_S$, where $u_S$ is the restriction of $u$ to $S$.
\end{maintheorem}

\bigskip

To explain how to reconstruct the eigenvectors of
$\Aop$ from the eigenvectors of the reduced matrix $\Rop_S(\lambda_0)$ we need the following concept of depth of a vertex $i\in V$.

\bigskip

\begin{definition}
\label{finite:depth} \,
The  {\em depth} of a vertex  $i\in V$ is defined recursively as follows.
\begin{enumerate}
\item A vertex $i\in S$ has depth $0$.
\item A vertex $i\in\Sc$ has depth $n$ iff\,  
$i$ has no depth less than $n$, and 
$w(i,j) \neq 0$ implies $j$ has depth $<n$, for all $j\in V$.
\end{enumerate}
\end{definition}

\bigskip

We denote by $S_n$ the set of all vertices of depth $\leq n$.
Because $S$ is a  structural set, every vertex $i$ has a finite depth.

If $\lambda_0 \in \C \setminus \Sigma$ is an eigenvalue of 
$\Aop$,
by Theorem~ \ref{finite:reduction} it is also an eigenvalue of the reduced matrix $\Rop_S(\lambda_0)$.
Knowing the eigenvector $u_S$ of this reduced matrix,
we can recover the corresponding eigenvector of $\Aop$  as follows:

\bigskip

\begin{proposition}[{\cite[Proposition 2.1]{DT}}]
\label{finite:reconstruction} \,
If $\lambda_0 \in \C \setminus \Sigma$ is an eigenvalue of 
$\Aop$  and $u_S=(u^{S}_{i})_{i\in S}$ is an eigenvector of the reduced matrix
$\Rop_S(\lambda_0)$  then the following recursive relations 
\begin{equation}\label{finite:recurs}
\left\{\begin{array}{l}
\smallskip\
u_i  = u^ {S}_{i}  \quad \text{ for } \;  i\in S_0=S\\ \\
\smallskip\
u_\ell  = {\displaystyle \sum_{j\in S_{n-1}} {\frac{w(\ell,j)}{\lambda_0-w(\ell,\ell)}\,u_{j}} } 
\quad \text{ for all }\;
\ell\in S_n\setminus S_{n-1} 
\end{array}\right.
\end{equation}
uniquely determine an eigenvector $u$ of $\Aop$ associated
with  $\lambda_0$.
\end{proposition}

\bigskip

Denote by $\Pi_S \colon \C^V\to\C^S$ the $S$-restriction  projection, and let $\Phi_S \colon \C\setminus \Sigma\to  \Mat_{V\times S}(\C)$ be the function that to each  $\lambda\in \C\setminus \Sigma$  associates the {\em reconstruction operator} $\Phi_S(\lambda) \colon \C^S \to \C^V$ where $u = \Phi_S(\lambda) v $ is recursively defined by
\begin{equation}\label{finite:recurs:operator}
\left\{\begin{array}{l}
\smallskip\
u_i  = v_{i}  \quad \text{ for } \;  i\in S_0=S\\ \\
\smallskip\
u_\ell  = {\displaystyle \sum_{j\in S_{n-1}} {\frac{w(\ell,j)}{\lambda-w(\ell,\ell)}\,u_{j}} } 
\quad \text{ for all }\;
\ell\in S_n\setminus S_{n-1} 
\end{array}\right. .
\end{equation}
These maps are inverse of each other in the sense that
$\Pi_S\circ\Phi_S(\lambda)={\rm id}_{\C^S}$
for all $\lambda\in\C\setminus \Sigma$.
Finally notice that the reconstruction operator
$\Phi_S(\lambda)$ is analytic in $\lambda\in\C\setminus\Sigma$.

\bigskip

The aim of this paper is to extend the reduction statements in Theorem~\ref{BW thm}, Theorem~\ref{finite:reduction} and Proposition~\ref{finite:reconstruction} to  infinite dimensional models.

\bigskip


\section{Infinite dimensional models }
\label{infinitemodel}

In this section we generalize the isospectral graph reduction theory 
to a class of bounded operators acting on $L^1$-spaces, i.e., Banach spaces of  integrable functions. 

Our infinite dimensional models will be defined by data tuples $(V,\FF,\mu, K,d,S)$ where  $(V,\FF,\mu)$ is a measure space, with $\mu$ being a positive $\sigma$-finite measure on $V$,
$K$ is a complex kernel on $V$, $d\colon V \to \C$ is a bounded measurable function   and $S\subseteq V$ is a subset satisfying   appropriate assumptions, referred to in the sequel as a {\em structural set}.
A bounded operator $\Aop\colon L^1(V,\mu)\to L^1(V,\mu)$ is defined by the data $(V,\FF,\mu, K,d)$ while the structural set $S\subseteq V$ 
  determines a {\em reduced operator} $\Rop_S\colon L^1(S)\to L^1(S)$ which will encapsulate the spectral behavior of $\Aop$.

\bigskip

Let  $L^1(V)=L^1(V,\mu)$ denote the Banach space of 
complex $\mu$-integrable  functions with the usual $L^1$ norm 
$$\norm{f}_{1}:=\int_V \vert f\vert\, d\mu.$$
Sometimes we will write $\norm{f}_{L^1(V)}$ instead of $\norm{f}_1$ to emphasize the domain of $f$.
Also, let
$L^\infty(V)$ denote the commutative Banach algebra of  complex bounded $\FF$-measurable functions with the usual sup norm $$\norm{f}_{\infty}:=\sup_{x \in V} \abs{f(x)}.$$

Finally, let $L^{1,\infty}(V\times V,\mu)$ be the space of measurable functions $f\colon V\times V\to\C$ such that
$$ \norm{f}_{1,\infty}:= \sup_{y\in V} \int_V \abs{f(x,y)}\, \mu(dx) 
<+\infty. $$
The functional $f\mapsto \norm{f}_{1,\infty}$ is a seminorm. With it, the quotient of  $L^{1,\infty}(V\times V,\mu)$ by the subspace of measurable functions $f\colon V\times V \to \C$ such that $f(x,y)=0$ for $\mu$-almost every $x\in V$ and for all $y\in V$ becomes a Banach space. As usual we identify 
$L^{1,\infty}(V\times V,\mu)$ with this quotient space and consider $\norm{\cdot}_{1,\infty}$ to be a norm.

\bigskip

\begin{definition} \label{kernel} \,
A  {\em kernel} on $V$ is any function 
$K \colon V\times\FF\to\C$ such that
\begin{enumerate}
\item the function  $B\mapsto K(x,B)$, from $\FF$ to $\C$, is 
a complex measure for any $x\in V$;
\item the function  $x\mapsto K(x,B)$, from $V$ to $\C$, is $\FF$-measurable for any $B\in\FF$.
\end{enumerate}
\end{definition}

\bigskip

In particular a kernel $K$ determines a function
$K \colon V\to \mathcal{M}(V,\C)$ that to each $x\in V$ associates the measure $K_x$. The notation $\mathcal{M}(V,\C)$ stands for the 
space of complex measures on $(V,\FF)$.
We use the following notation for the integral of an $\FF$-measurable function $f\colon V\to\C$ w.r.t. $K_x$
$$\int_V f\, dK_x = \int_V f(y)\, K(x,dy) . $$
We also define the positive  kernel
$ \vert K\vert \colon V \times \FF \to [0,+\infty]$
$$\vert K \vert(x,B) =  \abs{K_x}(B), $$
where $\abs{K_x}(B)$ stands for the 
total variation of  $K_x$  on $B$.
 
\bigskip

\begin{definition}
\label{de kernel with no diagonal}
We say that a kernel $K$ {\em has no diagonal part on a set $B\in\FF$} when
 for all $z\in B$, $K(z, \{z\})=0$.
\end{definition}

\bigskip

\begin{definition} 
\label{boundedkernel} \,
We say that a kernel $K$ is {\em  $(1,\infty)$-bounded}  when there exists a function
$h\in L^{1,\infty}(V\times V,\mu)$ such that
for all $x\in V$ and $B\in\FF$,
$$  K(x,B) = \int_B h(x,y)\, \mu(dy). $$ 
\end{definition}

The density function $h\colon V\times V\to \C$ of the kernel $K$ will be denoted by ${dK}/{d\mu}$.
We topologize the space of $(1,\infty)$-bounded kernels on $V$ with the distance associated with the  $(1,\infty)$-norm of its density function
\begin{equation}
\label{kernel norm}
\norm{K}  :=   \left\lVert{\frac{dK}{d\mu} } \right\rVert_{1,\infty} .
\end{equation}

\bigskip

Given a  
 $(1,\infty)$-bounded kernel
 $K$ and a measurable function $d \in L^\infty(V)$,
consider the operator $\Aop=\Aop_{d,K} \colon L^1(V) \to  L^1(V)$,  
\begin{equation}\label{operator}
(\Aop f)(x):= d(x)f(x) + \int_V f(y)\, K(x,dy)\;. 
\end{equation}
When $d \equiv 0$, kernel $K$ determines an operator
$\Qop=\Qop_K \colon L^1(V) \to  L^1(V)$, 
\begin{equation}
\label{Markov operator}
(\Qop f)(x):=\int_V f(y)\,K(x,dy)\;,
\end{equation}  
that we will refer to as the {\em Markov operator} of $K$.

\bigskip

Given a Banach space $(\mathscr{B},\norm{.})$, we denote by $\mathcal{L}(\mathscr{B})$ the Banach algebra of bounded linear operators on $\mathscr{B}$.  Given $Q \in \mathcal{L}(\mathscr{B})$, the {\em operator norm} of $Q$ is defined by
$\norm{Q}:=\sup_{\norm{f}=1} \norm{Qf}$.
The following proposition is a simple observation.

\begin{proposition}\label{Aop:bounded}
If $K$ is $(1,\infty)$-bounded and $d\in L^\infty(V)$  then $\Aop_{d,K} \in \mathcal{L}(L^1(V))$. 
Moreover, 
$\Aop_{d,K}$
has operator norm
$$ \norm{\Aop_{d,K}} \leq  \norm{d}_\infty + \norm{K}  .$$
In particular also $\Qop_{K}\in \mathcal{L}(L^1(V))$.
\end{proposition}

\bigskip

 Throughout the rest of this section we assume that

\begin{enumerate}
\item[(A1)]
 $K$ is a  $(1,\infty)$-bounded kernel on $V$;
\item[(A2)] $d\in L^\infty(V)$.
\end{enumerate}

\bigskip

Given  $S \subseteq V$ we will write $\Sigma_d=\Sigma_d(S):=\overline{d(V\setminus S)}$ (where the overline means topological closure in $\C$). 
Consider on the $\sigma$-algebra $\FF_S:=\{ S\cap B\colon B\in\FF\}$,
 the induced measure $\mu_S\colon \FF_S\to [0,+\infty]$
 defined by
$\mu_S( B):=\mu(B)$. We will write   $L^1(S)=L^1(S,\mu_S)$.

\bigskip

Next we introduce the family of {\em reduced operators} on $L^1(S)$ at a formal level.
Given $\lambda \in \C \setminus \Sigma_d$, 
we define  
$\Rop_S(\lambda)=\Rop_{S,d,K}(\lambda) \colon L^1(S) \to L^1(S)$ by

\begin{equation}\label{infinite:reduced:operator}
(\Rop_S(\lambda)\, f)(x):= d(x) f(x) + \int_S f(y) \, R_{S,\lambda}(x,dy)\;, 
\end{equation}
where, for $B \in \FF_S$,
\begin{equation}
\label{RS series}
R_{S,\lambda}(x,B):= \sum_{n=1}^\infty  K_{S,\lambda}^{(n)}(x,B) \;,
\end{equation}
with $K_{S,\lambda}^{(1)}(x,B) := K(x,B)$  and for $n\geq 2$,
$$ K_{S,\lambda}^{(n)}(x,B):= \int_{\Sc}\cdots \int_{\Sc}
\frac{K(x,dz_1)\,K(z_1,dz_2)\,\cdots\, K(z_{n-1},B)}{\prod_{p=1}^ {n-1} \left(\lambda-d(z_p) \right)}\;. $$

The reduced operator $\Rop_S(\lambda)$ may not be well defined if the series~\eqref{RS series} fails to converge. We will now define two concepts of structural set $S\subseteq V$ for which the operators $\Rop_S(\lambda)$  become well defined. 

\bigskip

Let $S \subseteq V$. Given $x\in V$, $B\in\FF$ and $n \geq 2$  define
the $S$-{\em taboo measure}
$$  \tau_{S,n}(x,B) =\tau_{S, n, K}(x,B):= \int_{S^c} \cdots \int_{S^c} K(x,dz_1) K(z_1,dz_2) \cdots K(z_{n-1},B) .
$$

\bigskip

Let us say that a sequence $t_n>0$ 
 converges to $0$
{\em super exponentially} when 
$$\lim_{n\to \infty}\frac{1}{n}\,\log t_n=-\infty. $$

\bigskip

\begin{definition} [Structural set of type A]
\label{structural:A} \,
A nonempty set  $S \subseteq V$ is called a {\em structural set of type A} for $K$ if and only if there exist a sequence \, $t_n$ \, converging to $0$ super exponentially and a non-negative function 
$M\in L^{1,\infty}(V\times V,\mu)$  such that  \,
$$\abs{\tau_{S,n}(x,B)}\leq t_n\, \int_B M(x,y)\,\mu(dy)$$ 
for all $x\in V$, $B\in\FF$ 
and $n\geq 2$. 
\end{definition}

\bigskip

We call point-set map on  $V$ to any
map $F \colon V\to \mathcal{P}(V)$, where $\mathcal{P}(V)$ stands for the power set of $V$.
We write $F \colon V\ \rightrightarrows V$ to express that $F$ is a point-set map on $V$. We define recursively the iterates of a point-set map $F$
setting $F^0(x):=\{x\}$ and  for all $n\geq 1$
and $x\in V$,
$$ F^n(x) = \cup \{ \, F(y)\, \colon \, y\in F^{n-1}(x)\,\} .$$

\bigskip

\begin{definition}[Structural set of type B]
\label{structural:B} \,
A nonempty set  $S \subseteq V$ is called a {\em structural set  of
type B} for $K$ if  and only if there exists a  measurable function $M\colon \Sc\to [0,+\infty)$ such that 
defining the point-set map 
\, $\varphi_{\Sc}: \Sc \rightrightarrows  \Sc$, 
$$
\varphi_{\Sc}(x):=
S^c \cap {\rm supp}(K_x)
$$
one has
\begin{enumerate}
\item for all $x\in \Sc$, there exists $n \in \N$ such that $(\varphi_{\Sc})^ n(x)=\emptyset$,
\item  for all $x\in\Sc$ and $B\in\FF$, $B\subseteq \Sc$,
$$ \abs{K(x,B)}\leq M(x)\,\mu(B) \;, $$
\item setting $n_S \colon \Sc\to\N$,
$n_S(x):=\min\{\, k\in\N\,:\, (\varphi_{\Sc})^k(x)=\emptyset\,\}$ then
$$\int_{\Sc} n_S(x)\, M(x)\,\mu(dx)<+\infty \;. $$
\end{enumerate}
\end{definition}

\bigskip

\begin{remark}
The concepts of structural sets of type A and B are logically independent. 
\end{remark}

\begin{remark}\label{nodiagonal}
If $(V,K)$ admits a structural set of type A or B then $K$ has no diagonal part on $\Sc$.
Indeed, for structural sets of type A, notice that if $p:=\vert K(z,\{z\}) \vert >0$ with $z\in \Sc$ then
$\vert \tau_{S,n}(z,\{z\}) \vert \geq p^n$ for all $n\in\N$.
More generally, if for some $m\geq 1$
one has  $p:=\vert K^m(z,\{z\}) \vert >0$ with $z\in \Sc$   then
 $\vert \tau_{S,n\,m}(z,\{z\}) \vert \geq p^n$ for all $n\in\N$.
This means the kernel $K$ has no cycles in $\Sc$.
For structural sets of type B, it follows from Definition~\ref{structural:B} (1) that $K$ has no diagonal part on $\Sc$ and also no cycles in $\Sc$.
\end{remark}

\begin{remark}
If $V$ is finite and $S$ is a structural set of type A or B for $K$ then $S$ is a structural set in the sense of Definition~\ref{finite:structural}.
\end{remark}

\bigskip

\begin{definition}[Structural set of type A quasi-B]
\label{structural:AB} \,
A nonempty set  $S \subseteq V$ is called a {\em structural set  of
type A quasi-B}  for $K$ if  and only if  there exists a sequence of 
$(1,\infty)$-bounded kernels
$K_n$, $n \in \N$, such that 
\begin{enumerate}
\item $S$ is a structural set of type A for $K$;
\item $\displaystyle{\lim_{n \to + \infty}} \norm{K-K_n}=0$;
\item $S$ is a structural set  of type B for $K_n$, for all $n \in \N$.
\end{enumerate}
\end{definition}

\bigskip

\begin{proposition}
Assume $V$, $K$ and $d \colon V\to \C$  satisfy (A1)-(A2) and $S$ is a structural set of type A quasi-B  for $K$. Then\;
$\lim_{n\to +\infty} \Aop_{d,K_n}=\Aop_{d,K}$ in $\mathcal{L}(L^1(V))$. 
\end{proposition}

\begin{proof}
Let 
$K(x, dy) = h(x,y)\,\mu(dy)$
and
$K_n(x, dy) = h_n(x,y)\,\mu(dy)$. 
Given $f \in L^1(V)$ we have that
$$\begin{array}{rcl} 
\displaystyle{\left\lVert{ \Aop_{d,K} f- \Aop_{d,K_n} f}  \right\lVert_{1}} & = & \displaystyle{\int_V\,\sabs{(\Aop_{d,K} f -\Aop_{d,K_n} f) (x) }\,\mu (dx)} \\ \\
{} & = & \displaystyle{\int_V\,\sabs{\int_V (K(x,dy)-K_n(x,dy)) f (y) }\,\mu (dx)} \\ \\
{} & = &\displaystyle{\int_V\,\sabs{\int_V (h(x,y)-h_n(x,y))   \, f (y) \, \mu(dy)}\,\mu (dx)} \\ \\
{} & \leq & \displaystyle{\int_V\,\int_V \sabs{h(x,y)-h_n(x,y)} \sabs{f (y)} \,\mu(dx)\, \mu (dy)} \\ \\
{} & \leq &  \norm{h-h_n}_{1,\infty} \, \norm{f}_{1} =
\norm{K-K_n} \, \norm{f}_{1} . \\
\end{array}
$$

\medskip
\noindent
By Definition~\ref{structural:AB}(2) (of type A quasi-B structural set)  we have that $\norm{K-K_n} \to 0$. 
Thus, $\lim_{n\to +\infty} \Aop_{d,K_n}=\Aop_{d,K}$ in $\mathcal{L}(L^1(V))$.
\end{proof}

\bigskip

\subsection{Reduced operator}\label{Reduced:well:defined}

We shall now prove that for structural sets of type A or type B, the reduced operators $\Rop_S(\lambda)$ defined  by~\eqref{infinite:reduced:operator}  are well defined, bounded  and, moreover, that the function $\Rop_S \colon \C\setminus \Sigma_d \to \Lops(L^1(S))$  is analytic.

\bigskip

\begin{lemma}\label{conv}
Assume $V$, $K$ and $d \colon V\to \C$  satisfy (A1)-(A2) and $S$ is a structural set of type  A. Given $\lambda \in \C \setminus \Sigma_d$, the kernel series $\sum_{n=1}^\infty  K_{S,\lambda}^{(n)}(x,B)$ converges absolutely and uniformly on $S\times \FF_S$.
Moreover, the operators $\Rop_S(\lambda)$
 defined by~\eqref{infinite:reduced:operator}  are bounded.
\end{lemma}

\begin{proof} 
For each $r>0$, define the open set
$\Omega_r:=\{\lambda\in\C\colon 
\dist(\lambda,\Sigma_d)>r\}$, so that
$\C\setminus \Sigma_d=\cup_{r>0}\Omega_r$.
Given $\lambda\in\Omega_r$
and a list of points 
$\bar z= (z_1,\ldots, z_{n-1})\in (\Sc)^{n-1}$,
the analytic function
\begin{equation*}
 f_{\bar z}(\lambda):= \frac{1}{\prod_{p=1}^ {n-1} \left(\lambda-d(z_p)\right)}
\end{equation*}
is bounded by $r^{-(n-1)}$.  
Therefore, given $x\in S$ and $B\in\FF_S$
$$\abs{K_{S,\lambda}^{(n)}(x,B)}\leq 
\frac{\abs{ \tau_{S,n}(x,B)} }{r^{n-1}} \leq 
\frac{ t_{n}}{r^{n-1}}\,\int_B M(x,y)\,\mu(dy)  \,.$$
But since 
the sequence $t_n \searrow 0$
super exponentially,
applying d'Alembert's criterion (ratio test) we can conclude that 
$$C_r:= \sum_{n=1}^\infty \frac{ t_{n}}{r^{n-1}}<+\infty .$$
From the previous bound on $\abs{K_{S,\lambda}^{(n)}(x,B)}$ we infer that for all  $f\in L^1(S)$
$$ \int \abs{ \int f(y) \,   K_{S,\lambda}^{(n)}(x,dy) } \, \mu(dx) \leq \frac{ t_{n}}{r^{n-1}}\, \norm{f}_{L^1(S)} \,\norm{M}_{1,\infty}.$$
Hence the Markov operator defined by $K_{S,\lambda}^{(n)}$ is in $\Lops(L^1(S))$ with norm bounded by 
$\frac{t_{n}}{r^{n-1}}\, \norm{M}_{1,\infty}$.
A straightforward calculation shows that the reduced operator
$\Rop_S(\lambda)$ has norm 
$\norm{ \Rop_S(\lambda) }  \leq \norm{d}_\infty + C_r\,\norm{M}_{1,\infty}$,
 for $\lambda\in\Omega_r$.  
\end{proof}

\bigskip

\begin{proposition}\label{holoI}
Assume $V$, $K$ and $d \colon V\to \C$  satisfy (A1)-(A2) and $S$ is a structural set of type  A.
Then the  function $\Rop_S \colon \C\setminus \Sigma_d \to \Lops(L^1(S))$  is analytic.
\end{proposition}

\begin{proof}
Consider the open sets $\Omega_r$, $r>0$, and the
analytic function $f_{\bar z}(\lambda)$, 
$\bar z\in (\Sc)^{n-1}$,
introduced in the proof of Lemma~\ref{conv}.
Differentiating $f_{\bar z}(\lambda)$ in $\lambda$ we get
$$ f_{\bar z}'(\lambda)=
-\left( \prod_{p=1}^{n-1}  \frac{1}{\lambda-d(z_p)}\right) \, \sum_{p=1}^{n-1} \frac{1}{\lambda-d(z_p)} , $$
with  $\vert f_{\bar z}'(\lambda)\vert \leq n\,r^{-n}$,  for all $\lambda\in\Omega_r$.

For each $n\in\N$ consider the operator
$\mathscr{K}_n(\lambda) \colon L^1(S)\to L^1(S)$ defined by

\begin{align*}
(\mathscr{K}_n(\lambda)\, h)(x) &:=  \int_S  h(y) \, K_{S,\lambda}^{(n)} (x,dy) \\
&=  \int_S \int_{S^c}\cdots \int_{S^c}  h(y) \, f_{z_1,\ldots, z_{n-1}}(\lambda)\, K(x,dz_1)\,K(z_1,dz_2)\,\cdots\, K(z_{n-1},dy) .
\end{align*}
By the previous bound, arguing as in Lemma~\ref{conv},
these operators are analytic with
$$ \norm{\frac{d}{d\lambda}\mathscr{K}_n(\lambda)}
\leq n\, \frac{ t_{n}}{r^{n}} \, \norm{M}_{1,\infty}\,, $$
for all $\lambda\in\Omega_r$.
Thus, by d'Alembert's criterion,
 the series of analytic functions
$\lambda\mapsto \sum_{n=1}^\infty \frac{d}{d\lambda}\mathscr{K}_n(\lambda)\in \Lops\left(L^1(S)\right)$ converges uniformely on $\Omega_r$. This proves that
$\lambda \mapsto \Rop_S(\lambda) $
is analytic on $\C\setminus \Sigma_d$.
\end{proof}

\bigskip

\begin{lemma}\label{conv:bounded}
Assume $V$, $K$ and $d \colon V\to \C$  satisfy (A1)-(A2) and $S$ is a structural set of type  B.
Given $\lambda \in \C \setminus \Sigma_d$, the kernel series $\sum_{n=1}^\infty  K_{S,\lambda}^{(n)}(x,B)$ converges absolutely and uniformly on $S\times \FF_S$.
Moreover, the operators $\Rop_S(\lambda)$
 defined by~\eqref{infinite:reduced:operator}  are bounded.
\end{lemma}

\begin{proof} 
Define 
$\D_n=\{\, z\in \Sc\,:\, n_S(z)\geq n \,\}$, where $n_S(z)$ is the function introduced in Definition~\ref{structural:B}(3). Given $B\in\FF_S$ and $x\in S$ one has
$$ K_{S,\lambda}^{(n)}(x,B)  = \displaystyle{\int_{\D_1}\cdots \int_{\D_{n-1}}
\frac{K(x,dz_1)\,K(z_1,dz_2)\,\cdots\, K(z_{n-1},B)}{\prod_{p=1}^ {n-1} \left(\lambda-d(z_p) \right)}\,}.
$$

 Consider the open sets $\Omega_r$, $r>0$,  
 introduced in the proof of Lemma~\ref{conv}.
From (A1), arguing as in the proof of Lemma~\ref{conv},
for all $\lambda\in\Omega_r$ 
 we have that
\begin{align}\label{bound}
\int \abs{K_{S,\lambda}^{(n)}(x,B)}\,\mu(dx) &\leq  
\frac{\norm{K}^2}{r^{n-1}}
\, \left(\int_{D_1} M(z_1)\, \mu(dz_1) \right)  \, \cdots \,  \left(\int_{D_{n-2}} M(z_{n-2})\, \mu(dz_{n-2}) \right)  \,  \mu(B) \,
\nonumber
\\
&=
 \frac{\norm{K}^2}{r^{n-1}}
\,\norm{M}_{L^1(D_1)}   \, \cdots \,  \norm{M}_{L^1(D_{n-2})}  \, \mu(B)\,.
\end{align}
But since 
$$\sum_{n=1}^ \infty \int_{D_{n}} M\,d\mu = \int_{\Sc} n_S(x)\,  M(x)\,\mu(dx)<+\infty, $$
applying d'Alembert's criterion (ratio test) we  conclude that  
$$C_r:= \sum_{n=1}^\infty \frac{\norm{K}^2}{r^{n-1}}\, \left(\int_{D_1} M\, d\mu \right) \, \cdots \,  \left(\int_{D_{n-2}} M\, d\mu \right) <+\infty. $$
A straightforward calculation shows that the operator
$\Rop_S(\lambda)\colon L^1(S) \to L^1(S)$ has norm 
$\norm{ \Rop_S(\lambda) } 
\leq \norm{d}_\infty + C_r$, for $\lambda\in\Omega_r$.
\bigskip

\end{proof}

\bigskip

\begin{proposition}\label{holoI:bounded}
Assume $V$, $K$ and $d \colon V\to \C$  satisfy (A1)-(A2) and $S$ is a structural set of type  B.
Then the function $\Rop_S \colon \C\setminus \Sigma_d \to \Lops(L^1(S))$  is analytic.
\end{proposition}

\begin{proof}
This proof is very similar to the one of Proposition~\ref{holoI}.
Consider the open sets $\Omega_r$, $r>0$,   
introduced in the proof of Lemma~\ref{conv}
and the operators 
$\mathscr{K}_n(\lambda)$, $n \in \N$, 
introduced in the proof of Proposition~\ref{holoI}.
By the previous bound in~(\ref{bound}), arguing as in the proof of Proposition~\ref{holoI},
these operators are analytic with
$$ \norm{\frac{d}{d\lambda}\mathscr{K}_n(\lambda)} 
\leq n\, \frac{\norm{K}^2}{r^{n-1}}\, \norm{M}_{L^1(D_1)}\, \cdots \,  \norm{M}_{L^1(D_{n-2})} , 
$$
for all $\lambda\in\Omega_r$.
Thus, by d'Alembert's criterion,
 the series of analytic functions
$\lambda\mapsto \sum_{n=1}^\infty \frac{d}{d\lambda}\mathscr{K}_n(\lambda)\in \Lops\left(L^1(S)\right)$ converges uniformely on $\Omega_r$. This proves that
$\lambda \mapsto \Rop_S(\lambda) $
is analytic on $\C\setminus \Sigma_d$.
\end{proof}

\bigskip

\subsection{ Type B isospectral reduction }
In this subsection we prove a isospectral reduction theorem for structural sets of type B.

For such  structural sets  the following key definition  captures the idea of depth  of a point in $V$ (see  Definition~\ref{finite:depth}).

\begin{definition}
\label{infinite:depth} \,
Assume $V$, $K$ and $d \colon V\to \C$  satisfy (A1)-(A2) and $S$ is a structural set of type  B.
We define, recursively in $n\in\N$, a measurable subset $S_n\subseteq V$ which can be regarded as the set of states $x \in V$ of depth  $\leq n$.
\begin{enumerate}
\item $S_0:=S$,
\item $S_n:=\{ x\in V \colon  \vert K\vert (x,V\setminus S_{n-1})=0 \}\cup S_{n-1}$.
\end{enumerate}
\end{definition}

\bigskip

\begin{proposition}
\label{prop Sn cover}
The sets $S_n$ cover $V$, i.e., \, $\displaystyle V= \cup_{n\geq 0}  S_n$.
\end{proposition}

\begin{proof}
Given $x\in V$, let $n=n_S(x)$
be as in Definition~\ref{structural:B}(3).
We will prove by induction 
that $\varphi_{\Sc}^{n-i}(x)\subseteq S_i$ for all  $i= 1,\dots, n$. 
Then, taking $i=n$, this shows that $\{x\}=\varphi_{\Sc}^0(x)\subseteq S_n$.

Let $y\in \varphi_{\Sc}^{n-1}(x)$.
Then $\varphi_{\Sc}(y)={\rm supp}(K_y) \cap \Sc=\emptyset$, which implies $\vert K\vert(y,\Sc)=0$,
and hence $\vert K\vert(y,V\setminus S_0)=\vert K\vert(y,\Sc)= 0$. Thus $y\in S_1$,
which proves the claim for $i=1$.

Assume, by induction hypothesis, that
$\varphi_\Sc^{n-i+1}(x)\subseteq S_{i-1}$, and let $y\in \varphi_\Sc^{n-i}(x)$, which implies that 
$\varphi_\Sc(y)\subseteq \varphi_\Sc^{n-i+1}(x)\subseteq S_{i-1}$.

We consider two cases:

First assume that $\mu(\varphi_\Sc(y))=0$. In this case, because $K_y$ is absolutely continuous w.r.t. $\mu$, we have that $\vert K\vert (y,\Sc)=0$.
Thus $\vert K \vert (y,V\setminus S_0) = \vert K \vert (y,\Sc) = 0$, and consequently $y\in S_1\subseteq S_i$.

Assume now that $\mu(\varphi_\Sc(y))>0$.
Since $\vert K(y,z)\vert >0$ for $\mu$-a.e. $z\in \supp( K_y)$,  and
$\varphi_\Sc(y)\subseteq S_{i-1}$, we have that
$\supp(K_y)\subseteq S_{i-1}$ and so
$\vert K \vert (y,V\setminus S_{i-1}) = 0$.  Therefore $y\in S_i$.

This proves that $\varphi_\Sc^{n-i}(x)\subseteq S_{i}$, and concludes the inductive argument.
\end{proof}

\bigskip

Given $u \in L^1(V)$, we define the functions $u_{S}$ and $u_{\Sc}$ in $L^1(V)$,
$$ u_{S}(x):=\left\{ 
\begin{array}{crr}
u(x) & \text{ if } & x\in S\\
0 & \text{ if } & x\in \Sc
\end{array}\right.\qquad
 u_{\Sc}(x):= \left\{ \begin{array}{crr}
u(x) & \text{ if } & x\in \Sc\\
0 & \text{ if } & x\in S
\end{array}\right. .$$
From now on, whenever appropriate we will identify $u_S$ with a function in $L^1(S)$.

\bigskip

\begin{definition}
The {\em spectrum} of the family of operators
$\Rop_S(\lambda)$, denoted by
$\spec(\Rop_S)$,  is the set of $\lambda\in \C\setminus\Sigma_d$ such that $\Rop_S(\lambda)-\lambda\,I$ is a non invertible
operator on $L^1(S)$.
\end{definition}

\bigskip

\begin{theorem}\label{main:B}
Assume $V$, $K$ and $d \colon V\to \C$  satisfy (A1)-(A2) and $S$ is a structural set of type  B.
Then
\begin{enumerate}
\item $ \spec(\Aop)\setminus \Sigma_d \subseteq \spec(\Rop_S) $.
\item Given $\lambda_0\in \C\setminus\Sigma_d$,
$\lambda_0$ is an eigenvalue of $\Aop$ iff 
$\lambda_0$ is an eigenvalue of $\Rop_S(\lambda_0)$.
\item If $\lambda_0 \in \C \setminus \Sigma_d$   is an eigenvalue of $\Aop$  and $u\in L^1(V)$ is an associated eigenfunction,
$\Aop\, u=\lambda_0\, u$, then  
$\Rop_S(\lambda_0)\, u_S=\lambda_0\, u_S$, i.e.,
$u_S$ is the corresponding eigenfunction for $\Rop_S(\lambda_0)$.

\item If $\lambda_0\in \C \setminus \Sigma_d$ is an eigenvalue of $\mathscr{R}_S(\lambda_0)$  and $v$ is an associated eigenfunction,
$\Rop_S(\lambda_0)\, v=\lambda_0\, v$,  then the following recursive relations 
\begin{equation}\label{recurs:I}
\left\{\begin{array}{l}
u(x)  =  v(x)  \quad \text{ for } \;  x\in S_0=S\\ \\
{\displaystyle u(z)  = \int_{ S_{n-1}} {\frac{K(z,dy)}{\lambda_0 - d(z)}\,u(y)} } 
\quad \text{ for all }\;
z\in S_n\setminus S_{n-1} 
\end{array}\right.
\end{equation}
uniquely determine an eigenfunction $u\in L^1(V)$ 
such that
$\Aop\, u=\lambda_0\, u$.
\end{enumerate}
\end{theorem}

\bigskip

\begin{remark}
\label{spec=}
Under the assumptions of Theorem~\ref{main:B}, if $S$ is finite then the equality holds in item (1), i.e., 
$ \spec(\Aop)\setminus \Sigma_d = \spec(\Rop_S) $.
Indeed, in this case, the reduced operator is finite dimensional and, consequently, the equality follows from item (2).
\end{remark}

\bigskip

To prove Theorem~\ref{main:B} we need the following lemma.

\bigskip

\begin{lemma}
\label{L1}
Assume $V$, $K$ and $d \colon V\to \C$  satisfy (A1)-(A2) and $S$ is a structural set of type  B.
Given $\lambda_0\in\C\setminus \Sigma_d$,
$f\in L^1(V)$ and $v\in L^1(S)$, 
defining recursively a function $u \colon V\to \C$ by
\begin{equation}\label{recurs:II}
\left\{\begin{array}{l}
u(x)  = v(x)  \quad \text{ for } \;  x\in S_0=S\\ \\
{\displaystyle u(z)  = -\frac{ f(z)}{\lambda_0-d(z)} + \int_{ S_{n-1}} {\frac{K(z,dy)}{\lambda_0 - d(z)}\,u(y)} } 
\quad \text{ for all }\;
z\in S_n\setminus S_{n-1} 
\end{array}\right.
\end{equation}
then $u\in L^1(V)$. Furthermore, 
$\norm{u}_{L^1(V)} \leq C (\norm{f}_{L^1(V)}  + \norm{v}_{L^1(S)}),$
for some constant $C=C(\lambda_0)>0$.
\end{lemma}

Given $B\in\FF$, let us denote by $\ind{B}$ the indicator function of $B$. 

Given a function $v\colon S\to\C$ we denote by $\bar v$ the extension $\bar v \colon V\to \C$,
$$ \bar v (x)=\left\{\begin{array}{lll}
v(x) & \text{ if } & x\in S \\
0 & \text{ if } & x\notin S \\
\end{array} \right. . $$

Given $\lambda\in \C\setminus \Sigma_d$, we introduce the operator
$\Dop(\lambda) \colon L^1(V)\to L^1(V)$ defined by 
\begin{equation}
\label{Dop}
(\Dop(\lambda) h)(z):= \ind{\Sc}(z)\,\frac{h(z)}{\lambda-d(z)},
\end{equation}
which  clearly is a bounded operator.

To simplify notations we will write $\Qop$ instead of  $\Qop_K$,
and   $\Dop$ instead of $\Dop(\lambda)$ whenever $\lambda$ is fixed.

\begin{proof}[Proof of Lemma~\ref{L1}]
Take  $v\in L^1(S)$, $f\in L^1(V)$, $\lambda_0\in\C\setminus\Sigma_d$ and assume that $\lambda_0\in\Omega_r$ where the set $\Omega_r$ was introduced
in the proof of Lemma~\ref{conv}.

Notice that  by Definition~\ref{infinite:depth}, of the sets $S_n$, we can replace $S_{n-1}$ by $V$ in~\eqref{recurs:II}.

Thus equation~\eqref{recurs:II} implies that  on $V$ one has
\begin{equation}
\label{u=-Df+QDu}
 u = (\bar v-\Dop f)  +  \Dop \Qop u . 
\end{equation}
We claim that  $I-\Dop\,\Qop$ is an invertible operator.

A simple calculation shows that for $z\in\Sc$
\begin{equation}
\label{u decomp}
 ((\Dop \Qop)^n h)(z)   =  \int_{V} \int_{\Sc}\cdots \int_{\Sc}
\frac{K(z,dz_1)\,K(z_1,dz_2)\,\cdots\, K(z_{n-1},d z_n)\, h(z_n)}{(\lambda_0-d(z))\prod_{p=1}^ {n-1} \left(\lambda_0-d(z_p) \right)  }\\
\end{equation}
with $((\Dop \Qop)^n h)(z)=0$ whenever $z\in S$.

As before let 
$\D_n=\{\, z\in \Sc\,:\, n_S(z)\geq n \,\}$, where $n_S(z)$ is the function introduced in Definition~\ref{structural:B}(3),
and set $D_0=S$.

Notice that for $z\in D_m$ when integrating  a function  over
$(z,z_1,\ldots, z_n)\in V\times (\Sc)^n$ against the kernel
$K(z,dz_1)\,K(z_1,dz_2)\,\cdots\, K(z_{n-1},d z_n)$  its integral vanishes outside the domain
$D_m\times D_{m-1}\times \ldots\times D_{m-n}$.

Hence, for $n_S(z)<n$ we have $((\Dop \Qop)^n h)(z)=0$.
Similarly, if $n_S(z)=n+j$ then

\begin{align*}
((\Dop \Qop)^n h)(z) & =  \int_{D_j} \int_{\D_{j+1}}\cdots \int_{\D_{n+j-1}}
\frac{K(z,dz_1)\,K(z_1,dz_2)\,\cdots\, K(z_{n-1},d z_n)\, h(z_n)}{(\lambda_0-d(z))\prod_{p=1}^ {n-1} \left(\lambda_0-d(z_p)  \right)} .
\end{align*}

Arguing as in the proof of Lemma~\ref{conv:bounded}
we obtain
$$ \norm{(\Dop \,\Qop)^n h}_{L^1(V)} \leq
\frac{1}{r^{n-1}}\, \norm{M}_{L^1(D_1)}\, \cdots \,  \norm{M}_{L^1(D_{n-1})}\,\norm{M}_{L^1(\Sc)}\,\norm{h}_{L^1(V)} . $$
From the convergence of the series
$\sum_{n=1}^\infty\frac{1}{r^{n-1}}\, \norm{M}_{L^1(D_1)}\, \cdots \,  \norm{M}_{L^1(D_{n-1})}$
 (see the proof of  Lemma~\ref{conv:bounded}) the claim follows.

Therefore, from~\eqref{u decomp} we get
$u=(I-\Dop\,\Qop)^{-1}(\bar v-\Dop f)$, which implies that

$$ \norm{u}_{1 }
\leq \max\{1,\norm{ \Dop}\}\,\left(\sum_{n=0}^\infty \norm{(\Dop\,\Qop)^n} \right)\,(\norm{\bar v}_{1} + \norm{f}_{1} ) .$$
\end{proof}

\begin{lemma}
\label{reconst:eq}
Let $V$, $K$ and $d \colon V\to \C$  satisfy (A1)-(A2) and $S$ be a structural set of type  B.
Then given $\lambda_0\in\C\setminus\Sigma_d$,
$f\in L^1(V)$ and $v\in L^1(S)$ 
the following two statements are equivalent:
\begin{enumerate}
\item $\displaystyle u = (\bar v-\Dop(\lambda_0) \,f)  +  \Dop(\lambda_0)\, \Qop \,u $,
\item $u=v$ on $S$ \, and \, $\displaystyle (\Aop-\lambda_0\,I)\, u = f$ on $\Sc$.
\end{enumerate}
\end{lemma}

\begin{proof}
For $z\in S$ we have $u(z)=v(z)$ whenever $u$ satisfies either (1) or (2).

For $z \in \Sc$, equation
$$\displaystyle u(z) = (\bar v-\Dop(\lambda_0) \,f)(z)  +  (\Dop(\lambda_0)\, \Qop \,u)(z) $$
is equivalent to
$$\displaystyle u(z) =   -\frac{ f(z)}{\lambda_0-d(z)} + \int_{ V} {\frac{K(z,dy)}{\lambda_0 - d(z)}\,u(y)} $$
which in turn is equivalent to
$$((\Aop-\lambda_0\,I)\, u)(z)= f(z).$$
\end{proof}

\begin{lemma}
\label{reconst:N}
Let $V$, $K$ and $d \colon V\to \C$  satisfy (A1)-(A2) and $S$ be a structural set of type  B.
Then given $\lambda_0\in\C\setminus\Sigma_d$, $f\in L^1(V)$ and
 $u\in L^1(V)$ such that $((\Aop-\lambda_0 I)\,u)(z)=f(z)$ for all $z\in \Sc$
the following two statements are equivalent:
\begin{enumerate}
\item $(\Rop(\lambda_0)-\lambda_0\, I) u_S =f_S$,
\item  $(\Aop-\lambda_0 I)\,u =f$ on $V$.
\end{enumerate}
\end{lemma}

\begin{proof}
We claim that if for all $z\in\Sc$,
$((\Aop-\lambda_0 I)\,u)(z) =0$, and $x\in S$ then 
\begin{equation}
\label{eq:N}
\int_{\Sc} u_{\Sc}(z)\,K(x,dz) = \sum_{p=2}^\infty
\int_{S} u_S(y)\, K_{S,\lambda_0}^{(p)}(x,dy) .
\end{equation}
Let us prove this claim.
Since $\Aop\, u=\lambda_0\, u$ on  $\Sc$,
we have for all $z\in\Sc$
$$
d(z) \, u_{\Sc}(z)  + \int_V u(w)\, K(z,dw) = \lambda_0\,u_{\Sc}(z)
$$
which is equivalent to
$$
d(z) \, u_{\Sc}(z)  + \int_S  u_S(y)\,K(z,dy) \, + \, \int_{\Sc} u_{\Sc}(z')\,K(z,dz') = \lambda_0\,u_{\Sc}(z).
$$
This in turn is equivalent to
\begin{equation}\label{infinite:iteration}
u_{\Sc}(z)=\frac{1}{\lambda_0 - d(z) } \int_S u_S(y)\,K(z,dy) \, + \, \frac{1}{\lambda_0 - d(z) } \int_{\Sc} u_{\Sc}(z')\, K(z,dz').
\end{equation}
Substituting ~$u_{\Sc}(z')$ by~(\ref{infinite:iteration}) in this relation we get
\begin{align*}
 u_{\Sc}(z) &= \frac{1}{\lambda_0 - d(z) } \int_S u_S(y)\,K(z,dy) +  \int_S  \int_{\Sc}
  u_S(y)\, \frac{K(z,dz')\,K(z',dy)}{(\lambda_0 - d(z))\,(\lambda_0 - d(z') )}  
 \\
 & \quad +    \int_{\Sc} \int_{\Sc}    u_{\Sc}(z'')\,\frac{K(z,dz') \, K(z',dz'')}{(\lambda_0 - d(z) )\,(\lambda_0 - d(z')) }  .
\end{align*}
Given $x\in S$, integrating in $z\in\Sc$ w.r.t. $K(x,dz)$ we obtain
\begin{align*}
\int_{\Sc}  u_{\Sc}(z)\, K(x,dz) &=   \int_S u_S(y)\,K_{S,\lambda_0}^{(2)}(x,dy) +  \int_S u_S(y)\,K_{S,\lambda_0}^{(3)}(x,dy) 
 \\
 & \quad +    \int_{\Sc} \int_{\Sc} \int_{\Sc}    u_{\Sc}(z'')\,\frac{K(x,dz)\,K(z,dz') \, K(z',dz'')}{(\lambda_0 - d(z) )\,(\lambda_0 - d(z')) }  .
\end{align*}

\medskip

\noindent
Proceeding inductively, we obtain for all $n \geq 1$ and 
$x\in S$,
\begin{align}
\int_{\Sc} u_{\Sc}(z)\, K(x,dz)  = &
\sum_{p=2}^n \, \int_S  \, u_S(y) \,  K_{S,\lambda_0}^{(p)}(x,dy)   \nonumber \\ 
 & + \,  \int_{\Sc}\cdots \int_{\Sc} u_{\Sc}(z_n) \,
\frac{K(x,dz_1)\,K(z_1,dz_2)\,\cdots\, K(z_{n-1},dz_n)}{\prod_{p=1}^ {n-1} \left(\lambda_0-d(z_p) \right)}. \label{structural:conv}
\end{align}
For $n\geq n_S(x)$ the remainder in~\eqref{structural:conv} vanishes which proves the claim.

\bigskip

Let us prove (1)$\Rightarrow$(2).
Assuming $\Rop_S(\lambda_0)\, u_S=\lambda_0\, u_S + f_S$
we have
\begin{align*}
f(x) + (\lambda_0 - d(x))\,u_S(x)   &= 
\sum_{p=1}^{\infty} \, \int_S  \, u_S(y) \,  K_{S,\lambda_0}^{(p)}(x,dy)\\
&= \int_S  \, u_S(y) \,  K(x,dy) + \sum_{p=2}^{\infty} \, \int_S  \, u_S(y) \,  K_{S,\lambda_0}^{(p)}(x,dy)\\
&= \int_S  \, u_S(y) \,  K(x,dy) + \int_{\Sc}  \, u_{\Sc}(y) \,  K(x,dy) = (\Qop \,u)(x),
\end{align*} 
where in the last step we use claim~\eqref{eq:N}.
This proves that $(\Aop-\lambda_0 I)\,u = f$ on $S$.
Since we are also assuming that $(\Aop-\lambda_0 I)\,u = f$ on $\Sc$, item (2) follows.

\bigskip

Let us now prove (2)$\Rightarrow$(1).
Assuming $(\Aop \,u-\lambda_0 I) \, u=f$ 
and using the claim~\eqref{eq:N},
we have for all $x\in S$,
\begin{align*}
f(x) + (\lambda_0-d(x))\, u(x) &= 
\int_S u_S(y)\, K(x,dy) + 
\int_{\Sc} u_{\Sc} (y)\, K(x,dy)\\
&= 
\int_S u_S(y)\, K(x,dy) + 
\sum_{p=2}^\infty
\int_{S} u_S(y)\, K_{S,\lambda_0}^{(p)}(x,dy) \\
&= \int_S u_S(y)\, R_{S,\lambda_0}(x,dy). 
\end{align*}
This proves that $(\Rop(\lambda_0)-\lambda_0 I) u_S =f_S$.
\end{proof}

\bigskip

We are now ready for the proof of Theorem~\ref{main:B}.

\begin{proof}[Proof of theorem~\ref{main:B}]

Item (3) and the direct implication in (2) are
consequences of Lemma~\ref{reconst:N} with $f=0$.

Next we prove (4) and the converse implication in (2).

Let $\lambda_0\in \C\setminus\Sigma_d$ be an eigenvalue of
$\Rop_S(\lambda_0)$ and $v\in L^1(S)$ be an associated eigenfunction,
$\Rop_S(\lambda_0)\, v=\lambda_0\,v$.
Applying Lemma~\ref{L1} with $f=0$
there exists a function $u\in L^1(V)$ defined recursively by~\eqref{recurs:II}. 
Observe that~\eqref{recurs:I} is the same as~\eqref{recurs:II}
when $f=0$.
As noticed in the proof of Lemma~\ref{L1},~\eqref{recurs:II} is equivalent to ~\eqref{u=-Df+QDu}, which in this case takes the form
$ u = \bar v + \Dop\,\Qop u $.
In  turn, by Lemma~\ref{reconst:eq},
this equation implies that   $(\Aop -\lambda_0 I )\, u= 0$ on $\Sc$.
Now, since $\Rop_S(\lambda_0)\, v=\lambda_0\,v$ and $(\Aop -\lambda_0 I )\, u= 0$ on $\Sc$, by Lemma~\ref{reconst:N}, with $f=0$, we obtain that $(\Aop -\lambda_0 I )\, u= 0$ on $V$.

Finally we prove (1).

Given $\lambda_0\notin \Sigma_d$, assume  that $\lambda_0\notin \spec(\mathscr{R}_S)$.
Then, given $f\in L^1(V)$, there exists $v\in L^1(S)$ such that 
$(\mathscr{R}_S(\lambda_0)-\lambda_0\,I)\,v=f_S$,
where $f_S\in L^1(S)$ stands for the restriction of $f$ to $S$.

Consider by Lemma~\ref{L1} the  function $u\in L^1(V)$
defined recursively by~\eqref{recurs:II}.
As noticed above,~\eqref{recurs:II} is equivalent to ~\eqref{u=-Df+QDu},   which once more, by Lemma~\ref{reconst:eq},
is equivalent to $(\Aop -\lambda_0 I )\, u= f$ on $\Sc$.
Now, since $(\mathscr{R}_S(\lambda_0)-\lambda_0\,I)\,v=f_S$ and $(\Aop -\lambda_0 I )\, u= f$ on $\Sc$, by Lemma~\ref{reconst:N} we obtain that $(\Aop -\lambda_0 I )\, u= f$ on $V$.

On the other hand, by (2) this operator must be injective.
Therefore $\lambda_0\notin \spec(\Aop)\setminus\Sigma_d$.
\end{proof}

\bigskip

Given a structural set $S$ of type B 
consider the map 
$$\Psi_S=\Psi_{S,d,K} \colon \C\setminus \Sigma_d\to \Lops\left(
L^1(V)\times L^1(S), L^1(V) \right)$$
 that to each $\lambda\in\C\setminus \Sigma_d$, $f\in L^1(V)$ and $v\in L^1(S)$ associates the unique function $u=\Psi_S(\lambda)(f,v)$
defined recursively by~(\ref{recurs:II}).
Lemma~\ref{L1} proves that this function is well defined.

\begin{remark}\label{eigen}
Given $\lambda_0\in\C\setminus \Sigma_d$,
 $f\in L^1(V)$ and $v\in L^1(S)$ such that
 $(\Rop(\lambda_0)-\lambda_0\,I)\,v=f_S$ 
then the function $ u=\Psi_S(\lambda_0)(f,v)$ satisfies $\displaystyle (\Aop-\lambda_0\,I)\, u= f $.
\end{remark}

\begin{proof}
By definition  $ u=\Psi_S(\lambda_0)(f,v)$ satisfies~\eqref{recurs:II}, which as noticed in the proof of Lemma~\ref{L1} is equivalent to~\eqref{u=-Df+QDu}. Hence, by Lemma~\ref{reconst:eq} one has
$(\Aop-\lambda_0 I)\,u=f$ on $\Sc$. Therefore, by Lemma~\ref{reconst:N} we have that $(\Aop-\lambda_0 I)\,u=f$ on $V$.
\end{proof}

\begin{remark}
Under the assumptions of Theorem~\ref{main:B}, the function 
$\Psi_S \colon \C\setminus \Sigma_d \to \Lops(L^1(V) \times L^1(S),L^1(V))$ is analytic.
\end{remark}

\bigskip

We introduce now the family  of {\em reconstruction operators}  
 $\Phi_S=\Phi_{S,d,K} \colon \C\setminus \Sigma_d \to \Lops(L^1(S),L^1(V))$ defined by
 $$\Phi_S(\lambda)(v) := \Psi_S(\lambda) (0,v) .$$

\begin{remark}
If $\lambda\in \spec(\Rop_S)$ is an eigenvalue of $\Rop_S(\lambda)$ with eigenfunction $u\in L^1(S)$ then 
$v=\Phi_S(\lambda)(u)$ is an eigenfunction of $\Aop$ 
associated  
with the same eigenvalue, i.e.,
$\Aop v= \lambda\, v$. This explains the `reconstruction' terminology.
\end{remark}

\bigskip

\subsection{ Type A isospectral reduction }
In this subsection we prove a isospectral reduction theorem for structural sets of type A.

\begin{theorem}\label{main:A}
Assume $V$, $K$ and $d \colon V\to \C$  satisfy (A1)-(A2) and $S$ is a structural set of type A. Let  $\lambda_0 \in \C \setminus \Sigma_d$  be an eigenvalue of $\Aop$ and 
$u \in L^1(V)$ be an associated eigenfunction, 
$\Aop u = \lambda_0 u$. 
Then $\lambda_0$ is also an eigenvalue of $\Rop_S(\lambda_0)$ and
$\Rop_S(\lambda_0)\, u_S=\lambda_0\, u_S$, i.e.,
$u_S$ is the corresponding eigenfunction for $\Rop_S(\lambda_0)$.
\end{theorem}

\begin{proof}
The proof follows the steps of the statement (2)$\Rightarrow$(1) in Lemma~\ref{reconst:N} with $f=0$.
For each $r>0$, consider the open set
$\Omega_r:=\{\lambda\in\C\colon 
\dist(\lambda,\Sigma_d)>r\}$, so that
$\C\setminus \Sigma_d=\cup_{r>0}\Omega_r$.
Let $\lambda_0 \in \Omega_r$.
We just need to observe that, because $S$ is a structural set of type A, the reminder integral in equality~(\ref{structural:conv}) converges to $0$ as $n\to +\infty$. 
Actually, the 
equality~\eqref{eq:N}
 holds
regardless of the structural set type. 
Indeed,
the integral
$$ \int_S \sabs{ \int_{\Sc}\cdots \int_{\Sc} u_{\Sc}(z_n) \,
\frac{K(x,dz_1)\,K(z_1,dz_2)\,\cdots\, K(z_{n-1},dz_n)}{\prod_{p=1}^ {n-1} \left(\lambda_0-d(z_p) \right)} }\,\mu(dx) $$
is bounded by 
$\frac{t_n}{r^{n-1}}\,\int_S \int_{\Sc}  \sabs{u_{\Sc}(z_n)}\,M(x,z_n) \,\mu(dz_n)\,\mu(dx)\leq  \frac{t_n}{r^{n-1}}\,\norm{M}_{1,\infty}\,\norm{u}_1$,
 and hence
converges to $0$ as $n\to +\infty$.
\end{proof}

\bigskip

\subsection{ Type A quasi-B isospectral reduction  }
In this subsection we prove a isospectral reduction theorem for structural sets of type A quasi-B.

\bigskip

Let  $(E,\norm{.})$ be a a Banach space and let 
$P \in  \Lops(E)$. Let $B_1$ be the closed unit ball in $E$. The operator $P$ is called {\em weakly compact} if the weak closure of 
$P B_1$ is compact in the weak topology (see~\cite{DS}). The set of weakly compact operators is closed in the uniform operator topology of $\Lops(E)$~(\cite[\S VI. 4., Corollary 4]{DS}).

\bigskip

\begin{lemma}\label{wc:operators}
Let $\Aop \colon E\to E$ be a weakly compact linear operator
on a Banach space $E$,
and $\Aop_n \colon E\to E$ a sequence of linear operators converging to $\Aop$. Let $u_n\in E$ be a unit eigenvector of $\Aop_n$ with $\Aop_n u_n=\lambda_n\, u_n$ and assume $\lambda_0=\lim_{n\to \infty} \lambda_n$ is non zero. Then
\begin{enumerate}
\item $\lambda_0$ is an eigenvalue of $\Aop$.
\item The sequence $\{u_n\}_n$ is relatively compact.
\item Any sublimit $u$ of $\{u_n\}_n$ is an eigenvector of $\Aop$ with $\Aop\,u=\lambda_0\,u$.
\end{enumerate}

\end{lemma}

\begin{proof}
By spectrum continuity, $\lambda_0\in \spec(\Aop)$.
Since $\lambda_0\neq 0$ and $\Aop$ is weakly compact,
$\lambda_0$ is an isolated eigenvalue with finite muliplicity. Hence $\spec(\Aop)=\{\lambda_0\}\cup \Sigma$
for some compact set $\Sigma\subset\C$ with $\lambda_0\notin\Sigma$.
Let $E=F\oplus H$ be the corresponding $\Aop$-invariant decomposition where $F$ is the generalized eigenspace associated  
with $\lambda_0$.
Consider a simple closed, positively oriented curve $\Gamma$ which isolates $\lambda_0$ from $\Sigma$ 
and denote by $R(\Aop, z):=(z I-\Aop)^{-1}$ the resolvent of $\Aop$. Then the projection $P \colon E\to E$ onto $F$ parallel to $H$ is given by
$$ P =\frac{1}{2\pi i}\,\int_\Gamma R(\Aop, z)\,dz . $$
Of course for large $n$ the operators $\Aop_n$
admit a similar decomposition of the spectrum
$$ \spec(\Aop_n) = \Lambda_n \cup\Sigma_n, $$
where $\lambda_n\in\Lambda_n$, and $\Lambda_n$, $\Sigma_n$ are closed sets separated by $\Gamma$.
Hence  the operator
$$ P_n  =\frac{1}{2\pi i}\,\int_\Gamma R(\Aop_n, z)\,dz  $$
 is the projection onto an $\Aop_n$-invariant finite dimensional suspace $F_n$ (with same dimension as $F$). By definition it is clear that  $P_n\circ \Aop_n=\Aop_n\circ P_n$,
 which implies that $H_n:={\rm Ker}(P_n)$ is also $\Aop_n$-invariant. It also follows that $P_n$ converges to $P$.
 
 Now, since $\lambda_n\in\Lambda_n$, 
 we have $u_n=P_n u_n\in F_n$.
 The sequence  $\tilde u_n:= P u_n \in F$ 
 is relatively compact because $F$ is finite dimensional. On the other hand, since $P_n\to P$,
 we have that  $\norm{u_n-\tilde u_n}=\norm{P_n u_n- P u_n} \leq \norm{P_n-P}$ converges to $0$.
 Therefore $\{u_n\}_n$ is also relatively compact,
 which proves (2).
 
 Item (3) is clear.
\end{proof}

\bigskip

 It is well known (see e.g.,~(\cite[p. 104]{W})  that integral operators with an uniformly bounded kernel are weakly compact. Therefore, it follows from Proposition~\ref{Aop:bounded} and Lemma~\ref{conv} that:

\bigskip

\begin{proposition}
\label{Rhat:wc}
Assume $V$, $K$ with $d\equiv 0$  satisfy (A1).
Then
\begin{enumerate}
\item The operator $\Aop$ is weakly compact.
\item Given $\lambda \in \C\setminus\{0\}$, the reduced operators defined by~\eqref{infinite:reduced:operator}  are weakly compact.
\end{enumerate}
\end{proposition}
 
\bigskip

\begin{proposition}\label{conv:reduced_operators}
Assume $V$, $K$ and $d \colon V\to \C$  satisfy (A1)-(A2) and $S$ is a structural set of type  A quasi-B. Consider the sequence of  kernels $K_n$ in Definition~\ref{structural:AB}. Then for every
compact set $\Lambda\subset \C\setminus\Sigma_d$,
$$\lim_{n\to +\infty} \Rop_{S,d,K_n}(\lambda)=\Rop_{S,d,K}(\lambda)\; \text{ in } \; \Lops(L^1(S)),$$
uniformly on $\lambda\in\Lambda$.
\end{proposition}

\begin{proof}
Given  $\Lambda\subset \C\setminus \Sigma_d$ compact, choose $r>0$ so that 
$\Omega_r=\{\lambda\in\C\colon 
\dist(\lambda,\Sigma_d)>r\}$ contains $\Lambda$.

Let 
$K(x, dy) = h(x,y)\,\mu(dy)$ and
$K_n(x, dy) = h_n(x,y)\,\mu(dy)$.
 By Definition~\ref{structural:AB} (of type A quasi-B structural sets)  $\norm{h-h_n}_{1,\infty}\to 0$.

We need to compare $R^{(p)}_{S,K,\lambda}$ with $R^{(p)}_{S,K_n,\lambda}$. The density of the first kernel is
$$ h_{S,K,\lambda}^{(p)}(x,y):= \int_{\Sc}\cdots \int_{\Sc}
\frac{h(x,z_1)\,h(z_1,z_2)\,\cdots\, h(z_{p-1},y)}{\prod_{j=1}^ {p-1} \left(\lambda-d(z_j) \right)}\, \mu(dz_1)\ldots \mu(dz_{p-1})  ,$$
and a similar formula holds for the density of $R^{(p)}_{S,K_n,\lambda}$
with $h_n$ instead of $h$.

Write $\hat \tau_{K,i}:= \sup_{x\in V} \tau_{S,K,i}(x,V)$. 
Let $M\geq 0$, $M\in L^1(V,\mu)$ be a common upper bound such that
$\abs{h(x,y)}\leq M(x)$ and $\abs{h_n(x,y)}\leq M(x)$ for all $x,y\in V$ and $n\in \N$.
Then
$$\norm{h_{S,K,\lambda}^{(p)} - h_{S,K_n,\lambda}^{(p)} }_{1,\infty} \leq  \frac{ \norm{M}_1^{p-1}}{r^{p-1}}\,
\norm{h-h_n}_{1,\infty} \, \sum_{j=1}^{p-1} 
\hat \tau_{K, j}\, \hat \tau_{K_n, p-j} . $$
Then what we need to complete the proof is the following lemma.
\end{proof}

\bigskip
 
\begin{lemma}
If $\tau_n$  decays super exponentially then
 so does $\sum_{j=1}^n \tau_j \tau_{n-j}$.
\end{lemma}

\begin{proof}
It follows from the definition that $\tau_n$ decays super exponentially to $0$ if and only if for all $L>0$ there exists $C>0$
such that $\tau_n \leq C\, e^{-L n}$.

Hence, given $L>0$ there exists $C>0$ such that 
$\tau_n \leq C\, e^{-2 L n}$ for all $n\geq 1$.
Therefore
$$\sum_{j=1}^n \tau_j \, \tau_{n-j}  \leq 
C^2 \sum_{j=1}^n e^{-2Lj} \,e^{-2 L(n-j)} \leq n\, C^2\, e^{2 L n}
\leq C^2 \, e^{- L n}  $$
which proves that the sum $\sum_{j=1}^n \tau_j \, \tau_{n-j} $
 decays super exponentially to $0$.
\end{proof}

\bigskip
 
\begin{proposition}\label{conv:reconstruction_operators}
Assume $V$, $K$ and $d \colon V\to \C$  satisfy (A1)-(A2) and $S$ is a structural set of type  A quasi-B. Consider the sequence of  kernels $K_n$ in Definition~\ref{structural:AB}. Then  given any  compact set $\Lambda\subset \C\setminus \Sigma_d$,
the following limit exists  
\begin{equation}
\Psi_S(\lambda)=\Psi_{S,d,K}(\lambda) :=\lim_{n\to+\infty}
\Psi_{S,d,K_n}(\lambda)   
\end{equation}
with uniform convergence in $\lambda\in\Lambda$. 
\end{proposition}

\begin{proof}
We claim that for some $m\geq 1$, $\norm{(\Dop\,\Qop_K)^m} <1$.
 
From~\eqref{u decomp} in the proof of Lemma~\ref{L1},
we have 
\begin{align*}
((\Dop \Qop_K)^n h)(z) & =   \int_{\Sc}\cdots \int_{\Sc}
\frac{K(z,dz_1) \,\cdots\, K(z_{n-1},d z_n)\, h(z_n)}{(\lambda_0-d(z))\prod_{p=1}^ {n-1} \left(\lambda_0-d(z_p)  \right)}  
\end{align*}
 for all $z\in\Sc$, and $((\Dop \Qop_K)^n h)(z)=0$ whenever $z\in S$.

Arguing as in the proof of Lemma~\ref{conv}
we obtain for all  large $n$
$$ \norm{(\Dop \,\Qop)^n } \leq
\frac{t_n}{r^n}\,\norm{M}_{1,\infty} \ll 1 . $$
From the claim, and since by Definition~\ref{structural:AB}(2), 
$\lim_{n \to + \infty} \norm{K-K_n}_\infty=0$,
 we also have  $\norm{(\Dop\,\Qop_{K_n})^m} <1$ for all large enough $n$.  Hence the operators $I-\Dop\,\Qop_{K}$
and $I-\Dop\,\Qop_{K_n}$ are all invertible with  uniformly bounded inverses. In particular
$\lim_{n\to+\infty} (I-\Dop\,\Qop_{K_n} )^{-1}
= (I-\Dop\,\Qop_{K} )^{-1}$.
 
Given
$f\in L^1(V)$ and $v\in L^1(S)$, by the proof of Lemma~\ref{L1}, the `reconstructed' function $u_n=\Psi_{S,d,K_n}(\lambda)(f,v)$
is given by
$$u_n=(I-\Dop\,\Qop_{K_n})^{-1}(\bar v-\Dop f).$$
Therefore $(u_n)_n$ converges in $L^1$ to $u=(I-\Dop\,\Qop_{K})^{-1}(\bar v-\Dop f)$.
\end{proof}

\bigskip

The previous proposition allows us to define
the {\em  limit  reconstruction operators} as follows:
Given $\lambda\in \C\setminus \Sigma_d$,
$\Phi_S(\lambda) \colon L^1(S)\to L^1(V)$,
$$\Phi_S(\lambda)(v):= \Psi_{S}(\lambda)(0,v).$$

\bigskip

\begin{theorem}\label{main:AB}
Assume $V$, $K$ and $d \colon V\to \C$  satisfy (A1)-(A2) and $S$ is a structural set of type  A quasi-B.
Then
\begin{enumerate}
\item $ \spec(\Aop)\setminus \Sigma_d \subseteq \spec(\Rop_S) $.
\item Given $\lambda_0\in \C\setminus\Sigma_d$,
$\lambda_0$ is an eigenvalue of $\Aop$ iff 
$\lambda_0$ is an eigenvalue of $\Rop_S(\lambda_0)$.
\item If $\lambda_0 \in \C \setminus \Sigma_d$   is an eigenvalue of $\Aop$  and $u\in L^1(V)$ is an associated eigenfunction,
$\Aop\, u=\lambda_0\, u$, then  
$\Rop_S(\lambda_0)\, u_S=\lambda_0\, u_S$, i.e.,
$u_S$ is the corresponding eigenfunction for $\Rop_S(\lambda_0)$.

\item If $\lambda_0\in \C \setminus \Sigma_d$ is an eigenvalue of $\Rop_S(\lambda_0)$  and $v$ is an associated eigenfunction,
$\Rop_S(\lambda_0)\, v=\lambda_0\, v$,  then  
$u=\Phi_S(\lambda_0)(v)$ 
is an eigenfunction of $\Aop$, i.e.,
$\Aop\, u=\lambda_0\, u$.
\end{enumerate}
\end{theorem}

\begin{proof}
Since $S$ is a structural set of type A, by Theorem~\ref{main:A},
item (3) and the direct implication in (2) follow. The converse implication in (2) will follow from item (4).

Let us prove item (1).

Take $\lambda_0\notin \spec(\Rop_S)$.
This means that $\Rop_S(\lambda_0)-\lambda_0 I$ is an invertible operator. We are going to prove that $\Aop-\lambda_0 I$ is also invertible. By the direct implication in (2) we know that
$\Aop-\lambda_0 I$ is injective. Therefore, it is enough to show that 
$\Aop-\lambda_0 I$ is surjective.

To simplify notations we will write
$\Aop$, $\Aop_n$, $\Rop$, $\Rop_n$, $\Psi$ and $\Psi_n$
respectively instead of $\Aop_{S,d,K}$, $\Aop_{S,d,K_n}$, 
$\Rop_{S,d,K}$, $\Rop_{S,d,K_n}$, $\Psi_{S,d,K}$ and
$\Psi_{S,d,K_n}$.

Since $\lambda_0\notin \spec(\Rop)$ and by Proposition~\ref{conv:reduced_operators}  $\Rop_n$ converges to $\Rop$,
one has $\lambda_0\notin \spec(\Rop_n)$ for all large enough $n$.

Given $f\in L^1(V)$,  because $\Rop_n(\lambda_0)-\lambda_0 I$ is invertible, there exists $v_n\in L^1(S)$ such that 
$(\Rop_n(\lambda_0) -\lambda_0 I)\,v_n = f_S$.
The sequence $v_n$ is bounded because the operator $\Rop_n(\lambda_0)-\lambda_0 I$ is invertible.

By Proposition~\ref{Rhat:wc} the operator  $\Rop(\lambda_0)$ can be decomposed as  
$\Rop(\lambda_0)=d_S+ \widehat{\Rop}(\lambda_0)$ where $d_S$ is a diagonal operator and 
$\widehat{\Rop}(\lambda_0)$ is weakly compact.
Analogously, the operator  $\Rop_n(\lambda_0)$ is decomposed as  
$\Rop_n(\lambda_0)=d_S+ \widehat{\Rop}_n(\lambda_0)$, with the same diagonal part $d_S$   and where
$\widehat{\Rop}_n(\lambda_0)$ is also weakly compact.
Moreover, $\widehat{\Rop}_n(\lambda_0)$ converges to
$\widehat{\Rop}(\lambda_0)$ as $n\to+\infty$.

By weak compactness of $\widehat{\Rop}(\lambda_0)$, we can assume that
$\widehat{\Rop}(\lambda_0)\,v_n$ converges to some $w\in L^1(S)$. 
Since
\begin{align*}
 f_S -\widehat{\Rop}_n(\lambda_0)\,v_n  = (\Rop_n(\lambda_0)-\lambda_0 )\,v_n -\widehat{\Rop}_n(\lambda_0)\,v_n  =
 (d_S -\lambda_0 )\,v_n    
\end{align*}
and
\begin{align*}
\norm{\widehat{\Rop}_n(\lambda_0)\,v_n - w} &\leq 
\norm{\widehat{\Rop}_n(\lambda_0)\,v_n - \widehat{\Rop}(\lambda_0)\,v_n} +
\norm{\widehat{\Rop}(\lambda_0)\,v_n - w}\\
&\leq \norm{\widehat{\Rop}_n(\lambda_0) - \widehat{\Rop}(\lambda_0) }\,\sup_n \norm{v_n} 
+ \norm{\widehat{\Rop}(\lambda_0)\,v_n - w},
\end{align*}
we conclude that $(d_S -\lambda_0 )\,v_n $ converges to $f_S-w$, and hence 
$$\lim_{n\to +\infty} v_n =   (d_S-\lambda_0)^{-1}(f_S - w) =: v \, \text{ in } \, L^1.$$

By Remark~\ref{eigen} the function 
$u_n= \Psi_n(f,v_n)$ satisfies  
\begin{equation}
\label{An-l0}
(\Aop_n-\lambda_0 )\,u_n= f .
\end{equation}

On the other hand we have
\begin{align*}
\norm{\Psi_n(f,v_n)-\Psi(f,v)} &\leq 
\norm{\Psi_n(f,v_n)-\Psi(f,v_n)}
+ \norm{\Psi(f,v_n)-\Psi(f,v)}\\
&\leq 
\norm{\Psi_n-\Psi}  \,\sup_n \norm{v_n} 
+ \norm{\Psi}\,\norm{v_n -v}  
\end{align*}
which proves that $u_n= \Psi_n(f,v_n)$ converges to
$u=\Psi(f,v)$ in $L^1$.

Thus, taking the limit in~\eqref{An-l0} we get that $(\Aop-\lambda_0 )\,u= f$, which proves that $\lambda_0\notin \spec(\Aop)$.

Finally we prove (4).
 
Let $\lambda_0\in \C\setminus\Sigma_d$ be an eigenvalue of
$\Rop(\lambda_0)$ and $v\in L^1(S)$ be an associated eigenfunction,
$\Rop(\lambda_0)\, v=\lambda_0\,v$.
Since, by Proposition~\ref{conv:reduced_operators}, $\Rop_n$ converges to $\Rop$, 
there exist $(\lambda_n)_n$ satisfying
$\lambda_0 = \lim_{n \to +\infty} \lambda_n$ such that
$\lambda_n \in \spec(\Rop_n(\lambda_0))$.
By the uniformity of convergence in Proposition~\ref{conv:reduced_operators}, changing slightly the $\lambda_n$ if necessary, we may assume that 
$\lambda_n \in \spec(\Rop_n(\lambda_n))$.

Let $v_n \in L^1(S)$ be a unit eigenfunction of $\Rop_n(\lambda_n)$, i.e., $\Rop_n(\lambda_n) v_n = \lambda_n v_n$. 
Consider, as before, the weakly compact operators 
$\widehat{\Rop}(\lambda_0)$ and $\widehat{\Rop_n}(\lambda_n)$
so that $\Rop(\lambda_0)$ and $\Rop_n(\lambda_n)$ decompose as
$\Rop(\lambda_0)=d_S+ \widehat{\Rop}(\lambda_0)$ and 
$\Rop(\lambda_n)=d_S+ \widehat{\Rop_n}(\lambda_n)$.
Moreover, again by uniformity of convergence, $\widehat{\Rop}_n(\lambda_n)$ converges to
$\widehat{\Rop}(\lambda_0)$ as $n\to+\infty$.
By Lemma~\ref{wc:operators}, extracting a subsequence if necessary we can assume that 
$(v_n)_n$ converges to $v$.

Since $S$ is a structural set of type B for $K_n$, by Theorem~\ref{main:B}(4), there exists  a sequence of eigenfunctions $u_n \in L^1(V)$ such that
$$u_n=\Psi_n(\lambda_n)(0,v_n) \hspace{0.3cm} \mbox{and} \hspace{0.3cm} 
\Aop_n u_n = \lambda_n u_n.$$
Repeating the argument in the proof of item (1), now with $f =0$, and using uniformity of convergence in Proposition~\ref{conv:reconstruction_operators}, we obtain that
$u_n=\Psi_n(\lambda_n)(0,v_n)$ converges to 
$u=\Psi(\lambda_0)(0,v)=\Phi(\lambda_0)(v)$.
Hence $\Aop \, u=\lambda_0\,u$.
\end{proof}

\bigskip


\section{ Infinite graphs }
\label{countablegraphs}

In this section we specialize the theory in Section~\ref{infinitemodel} to countably infinite graphs with a finite structural set.
 We also propose a numerical algorithm to approximate the eigenfunctions of such graphs.
 
\bigskip

\begin{definition}\label{graph:countable} 
A  countable weighted graph is a pair
$G=(V,w)$ where $V$ is a countable set and 
$w \colon V\times V\to\C$ is any function, 
called the weight function of $G$.
\end{definition}

\bigskip

Assume $G=(V,w)$ is a countable weighted graph
over an infinite set $V$.  The  weight function $w \colon V\times V\to\C$
determines the following 
kernel $K(i,.)=\sum_{j \in V}  w(i,j)\,\delta_j(.)$,
where $\delta_j$ stands for the Dirac measure supported on $j$.

\bigskip

We define the Banach spaces
$$ L^1(V):=\{\, f \colon V\to\C\,\colon\,
\norm{f}_1:=\sum_{i\in V} \abs{f(i)} <+\infty\,\}\;, $$
and
$$ L^{1,\infty}(V \times V):=\{\, w \colon V \times V \to \C \,\colon\,
\norm{w}_{1,\infty}:=\sup_{j\in V}  \sum_{i\in V}
\abs{w(i,j)} <  +\infty\,\}\;. $$
Note that identifying the weight function $w$ with the kernel
$K(i,.)=\sum_{j \in V}  w(i,j)\,\delta_j(.)$
the norm $\norm{w}_{1,\infty}$ matches the one defined in~(\ref{kernel norm}).
 
 \bigskip
 
\begin{definition}\label{bounded:function} 
We say that the weight  function 
$w \colon V\times V\to\C$ is   {\em  $(1,\infty)$-bounded} when 
$w \in L^{1,\infty}(V \times V)$.
\end{definition}

\bigskip

Each $(1,\infty)$-bounded
function $w$ determines a Markov
operator $\Aop_w \colon L^1(V)\to L^1(V)$,
\begin{equation}\label{operator:III}
(\Aop_w f)(i):= \sum_{j\in V}  w(i,j)\, f(j) \;. 
\end{equation}

\bigskip

The following proposition is a simple observation.

\begin{proposition}
If $w$ is $(1,\infty)$-bounded
then $\Aop_w \in \Lops\left(L^1(V)\right).$
Moreover, 
$\Aop_{w}$
has operator norm
$$ \norm{\Aop_{w}} \leq  \norm{w}_{1,\infty}  .$$
\end{proposition}

\bigskip

\begin{theorem}\label{main:AB:countable}
Let $G=(V,w)$ be a countable weighted graph 
 and $S \subset V$ be a finite set.
Assume that:
\begin{enumerate}
\item[$\mbox{(i)}$] $w$ is $(1,\infty)$-bounded;
\item[$\mbox{(ii)}$] $S$ is a structural set of type A quasi-B for 
$w$ (in the sense of Definition~\ref{structural:AB}).
\end{enumerate}
Then
\begin{enumerate}
\item $ \spec(\Aop_w) \setminus \Sigma = \spec(\Rop_{S,w}) $.
\item Given $\lambda_0\in \C\setminus\Sigma$,
$\lambda_0$ is an eigenvalue of $\Aop_w$ iff 
$\lambda_0$ is an eigenvalue of $\Rop_{S,w}(\lambda_0)$.
\item If $\lambda_0 \in \C \setminus \Sigma$   is an eigenvalue of $\Aop_w$  and $u\in L^1(V)$ is an associated eigenfunction,
$\Aop_w\, u=\lambda_0\, u$, then  
$\Rop_{S,w}(\lambda_0)\, u_S=\lambda_0\, u_S$, i.e.,
$u_S$ is the corresponding eigenfunction for $\Rop_{S,w}(\lambda_0)$.
\item If $\lambda_0\in \C \setminus \Sigma$ is an eigenvalue of $\Rop_{S,w}(\lambda_0)$  and $v$ is an associated eigenfunction,
$\Rop_{S,w}(\lambda_0)\, v=\lambda_0\, v$,  then $u=\Phi_{S,w}(\lambda_0)(v)$ is an eigenfunction of $\Aop_w$, i.e.,
$\Aop_w\, u=\lambda_0\, u$.
\end{enumerate}
\end{theorem}

\begin{proof}
This theorem is a corollary of Theorem~\ref{main:AB}.
Notice that item (i) implies (A1), while (A2) is automatic since we are taking $d=0$.  Equality in item (1) holds because $S$ is finite (see Remark~\ref{spec=}).
\end{proof}

\bigskip

We propose now a numerical algorithm to approximate the eigenfunctions of 
a countably infinite graph.
The input and output of the algorithm will consist on the following:

\bigskip

\noindent
{Input:} \; 
\begin{enumerate}
\item[$\bullet$] a countable weighted graph $G=(V,w)$,
\item[$\bullet$] a finite set $S$,
\item[$\bullet$] a sequence $(w_n)_n$ of weight functions,
\item[$\bullet$] an integer $k$,
\item[$\bullet$] a finite subset $V_0$ such that $S\subseteq V_0\subseteq V$,
\end{enumerate}
where $G=(V,w)$, $S$ and 
$(w_n)_n$
satisfy the assumptions (i)-(ii) of Theorem~\ref{main:AB:countable}.
The weight functions $w_n$ are the $(1,\infty)$-bounded kernels in Definition~\ref{structural:AB}.
\bigskip

\noindent
{ Output:} \; an approximation of the $k$-th eigenvalue $\lambda$ of $\Aop_w$, and an approximation of the values $u(i)$  of a $\lambda$-eigenfunction $u$ for $\Aop_w$ computed at all vertices $i\in V_0$.

\bigskip

Now we describe the steps of the proposed algorithm.

\bigskip

\noindent
{Steps:} \; 
\begin{enumerate}
\item[(1)] Compute the $k$-th eigenvalue $\lambda_{k,n}$ of $\Aop_{w_n}$ for $n$ large, or, alternatively, compute the $k$-the zero of the analytic function $\det[ \Rop_{S,w_n}(\lambda)-\lambda\,I ]$.

\item[(2)] Compute  an  associated eigenvector $v_0$ of the 
finite dimensional matrix $\Rop_{S,w_n}(\lambda_{k,n})$, for some large $n$.

\item[(3)] 
Use the reconstruction operator 
$\Phi_{S,w_n}(\lambda_{k,n})(v_0)$
to obtain the wanted approximation.
\end{enumerate}

\bigskip


\section{A family of infinite Markov chains}
\label{Markov}

Consider a countable weighted graph $G=(V,w)$ such that 
$(w(i,j))_{i,j \in V} $ is a stochastic matrix.
More precisely assume $w(i,j)=p_{ij}$ is the transition probability from state $j$ to state $i$
of some Markov chain 
with infinite countable state space $V$. Note that in this case  $\lambda=1$ is an eigenvalue of the Markov operator $\Aop_w$. 
We remark that the (normalized) eigenvectors of $\Aop_w$, corresponding to the eigenvalue $\lambda=1$, are precisely the stationary measures of the given Markov process.

In this section we present an example where the theory developed is applied to give a closed formula for the stationary probability measures of a family of countable Markov chains.

\bigskip

Consider a Markov chain with state space $\N=\{1,2,\ldots\}$ and
transition probability matrix $(w(i,j))_{i,j \in \N}$ defined by
\begin{enumerate}
\item[(i)] \, $w(i,1)=a_i$, for all $i \in \N$;
\smallskip
\item[(ii)] \, $w(2,2)=1-b_1$;
\smallskip
\item[(iii)] \, $w(i-1,i)=b_{i-1}$, for all $i \geq 2$; \,
\smallskip
\item[(iv)] \, $w(1,i)=1-b_{i-1}$, for all $i \geq 3$; \,
 and 
\smallskip
\item[(v)] \, $w(i,j)=0 \hspace{0.2cm}$ otherwise,
\end{enumerate}
where 
$w(i,j)$ represents the transition probability from state $j$ to state $i$ (see Figure~\ref{markov:AB:finite}).
We assume that the 
transition probabilities $(a_i)_{i \in \N}$ and 
$(b_i)_{i \in \N}$ satisfy the following conditions:

\medskip

\begin{enumerate}
\item[(B1)] 
$\sum_{i=1}^\infty a_i=1 \hspace{0.2cm} \mbox{\text and} \hspace{0.2cm} 0 < a_i, b_i <1, \hspace{0.2cm} \mbox{\text for all} \,\, i \in \N\,;$ and

\smallskip

\item[(B2)]
there exist $C>1$ and $0<\rho<1$ such that $
b_i < C \rho^i\,$, for all $i \in \N$.
\end{enumerate}

We notice that condition (B2) implies that the sequence
$t_n:=\prod_{i=1}^{n-1} b_i$ converges to $0$ super
exponentially. Indeed, for all $n \in \N$,
$$\prod_{i=1}^{n-1} b_i < \prod_{i=1}^{n-1} C \rho^i  =C^{n-1}  \rho^{n(n-1)/2}\,,$$
which converges to $0$ super
exponentially.

\medskip

\begin{figure}[h]
\centerline{\includegraphics[width=0.6\textwidth]{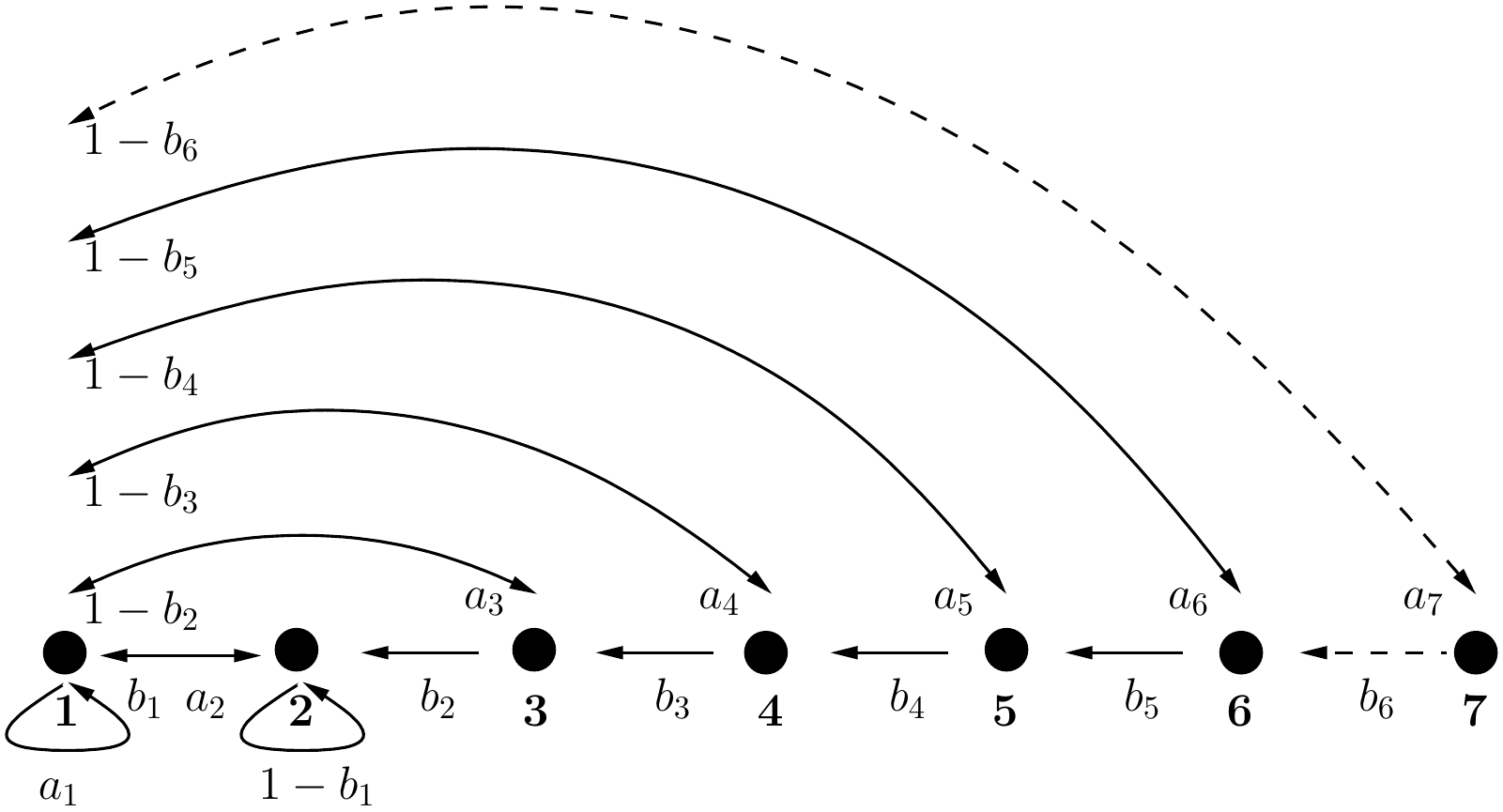}}
\caption{An infinite Markov chain.}
\label{markov:AB:finite}
\end{figure}

\begin{proposition}\label{prob:markov}
Consider a Markov chain with transition probability matrix 
$(w(i,j))_{i,j \in \N}$ defined by
$(i)$-$(v)$ and satisfying conditions (B1)-(B2).

This Markov chain has a unique stationary probability measure
$q=(q(i))_{i \in \N}$ given by
$$q(i)=\frac{u(i)}{\sum_{j=1}^\infty \left|u(j)\right|},$$
where 
$$ \left\{ \begin{array}{rclll}
u(i) &=&   v(i)  & \text{ if } & i =1,2\\
\\
u(i) &=&  
\sum_{k=1}^{\infty}\left(\prod_{\ell=0}^{k-2}b_{i+\ell}\right) a_{i+k-1} \, v(1) 
  & \text{ if } & i \geq 3\\
\end{array}\right.   $$
and $(v(1),v(2))$ is any eigenvector of the matrix
$$\Rop =\left[
\begin{array}{lc}
1 - \sum_{\ell=0}^{\infty} \left(\prod_{k=1}^{\ell} b_{k+1} \right) a_{\ell +2} & \hspace{0.2cm} b_1   \\[4pt] 

\sum_{\ell=0}^{\infty} \left(\prod_{k=1}^{\ell} b_{k+1} \right) a_{\ell +2} &  \hspace{0.2cm} 1-b_1    \\[4pt] 
\end{array}
\right]$$
associated with the eigenvalue $\lambda=1$.

\end{proposition}

\bigskip

The rest of this section is dedicated to the proof of this proposition.

\bigskip
The  matrix $(w(i,j))_{i,j \in \N}$ is stochastic in the sense that the sum of the entries of each column is $1$. 
This Markov chain is irreducible and aperiodic and hence admits a unique stationary probability measure.
The  weight function $w:\N\times \N \to [0,+\infty[$
determines the kernel $K(i,.)=\sum_{j \in\N} w(i,j)\,\delta_j(.)$, where $\delta_j$ stands for the Dirac measure supported on $j$.
To apply the previous results consider, as reference measure  $\mu$ on $\N$, the counting measure. 
Clearly, the weight function $w$ is $(1,\infty)$-bounded.

\bigskip

Consider the following sequence of Markov chains (see Figure~\ref{markov:AB:finite:app}) whose stochastic transition probability matrices $(w_n(i,j))_{i,j \in \N}$, $n \geq 2$, are defined by

\begin{enumerate}
\item[$\bullet$] \,
$w_n(i,1)=a_i$, for all $i \in \N$;\,
\smallskip
\item[$\bullet$] \, $w_n(2,2)=1-b_1$;
\smallskip
\item[$\bullet$] \, $w_n(i-1,i)=b_{i-1}$, for $i \in\{2,\ldots,n\}$; \,
\smallskip
\item[$\bullet$] \, $w_n(1,i)=1-b_{i-1}$, for $i \in\{3,\ldots,n\}$; \,
\smallskip
\item[$\bullet$] \, $w_n(1,i)=1$, for all $i > n$; \, and
\smallskip
\item[$\bullet$] \, $w_n(i,j)=0 \hspace{0.2cm}$ otherwise.
\end{enumerate}

\begin{figure}[h]
\centerline{\includegraphics[width=0.6\textwidth]{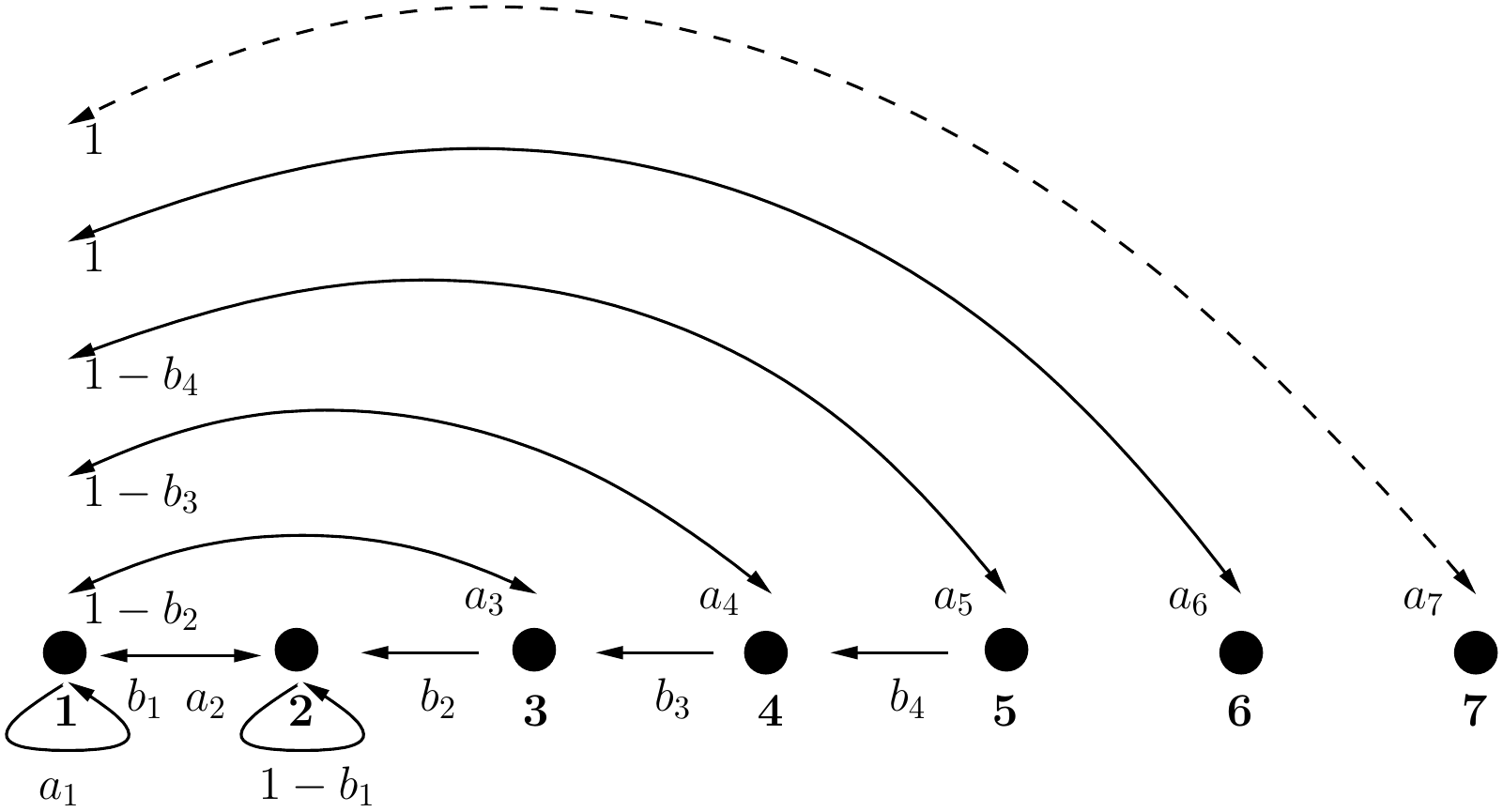}}
\caption{The Markov chain approximation $w_5$.}
\label{markov:AB:finite:app}
\end{figure}

\bigskip

Let $S=\{1,2\}$. We have that $S$ is a structural set of type A quasi-B for $w$ (in the sense of Definition~\ref{structural:AB}). Indeed,

\begin{enumerate}
\item $S$ is a structural set of type A for $w$ (in the sense of Definition~\ref{structural:A}).
Consider the function $M \in L^{1,\infty}(\N \times \N)$ defined by  
$M(i,j)=\rho^{i-1}$,
where $0<\rho<1$ is given by condition (B2).
Since the function $B\mapsto \tau_{S,n,w}(i,B)$ is a measure,
and taking in mind that the transition probabilities $(a_i)_{i \in \N}$ and $(b_i)_{i \in \N}$ satisfy conditions (B1)-(B2), we just need to observe that for all $n \geq 2$,

\bigskip

\begin{enumerate}

\item[$\bullet$] 
$\tau_{S,n,w}(1,1)= 
\sum_{\ell=1}^\infty 
\left(1-b_{\ell+1}\right)  \left(\prod_{k=1}^{n-2} b_{k+\ell+1}\right) a_{n+\ell}
\leq  \sum_{\ell=1}^\infty 
  \left(\prod_{k=1}^{n-2} b_{k+\ell+1}\right)$
  
  $\hspace{2cm} \leq C  \left(\rho^3+\frac{\rho^{n+1}}{1-\rho} \right) \left(\prod_{k=1}^{n-3} C \rho^k\right)\, M(1,1)$;

\medskip

\item[$\bullet$]  for $j > 1+n$,

$\tau_{S,n,w}(1,j) =
\left(1-b_{j-n}\right).\left(\prod_{k=1}^{n-1} b_{j-k}\right) \leq \prod_{k=1}^{n-1} b_{j-k} \leq  \left(\prod_{k=1}^{n-1} C \rho^k\right)\, M(1,j)$;

\medskip

\item[$\bullet$]  for $i\geq 2$,

$\tau_{S,n,w}(i,i+n)=
\prod_{k=0}^{n-1} b_{i+k} 
=\left(\prod_{k=1}^{n-1} b_{i+k} \right)b_i
\leq \left(\prod_{k=1}^{n-1} 
b_{i+k} \right) \! C \rho^i $

$\hspace{2.6cm} \leq  C\, \left(\prod_{k=1}^{n-1} C \rho^k\right) M(i,i+n)$; and

$\tau_{S,n,w}(i,1)=
\left(\prod_{k=0}^{n-2} b_{i+k}\right) a_{i+n-1} \leq
\prod_{k=0}^{n-2} b_{i+k}
\leq  C \left(\prod_{k=1}^{n-2} C \rho^k\right) M(i,1)$;

\medskip
\item[$\bullet$] 
$\tau_{S,n,w}(i,j) = 0$  \, in all other cases.
\end{enumerate}
\bigskip
Therefore, we can take in Definition~\ref{structural:A}  
$t_n:=C^2 \left(\rho^3+\frac{\rho^{n+1}}{1-\rho} \right)\left(\prod_{k=1}^{n-3} C \rho^k\right)$,
which converges to $0$ super exponentially.

\bigskip

\item $\lim_{n\to +\infty} \norm{w-w_n}_{1,\infty} = 0$. 
Observe first that, for all $n \geq 2$, 

$${(w(i,j)-w_n(i,j))}_{i,j \in \N}=\begin{pmatrix}
\begin{array}{ccccc:cccccc}
0 & {}  & \hdots &   {}  & 0   & -b_n &  -b_{n+1}&   -b_{n+2} &
-b_{n+3} & \hdots  {}  \\[4pt] 
{} & {} &  {} & {}  &  {}  &  0    & 0  &   0  & 0  & \hdots & {} \\
\vdots & {} & \ddots & {}  & \vdots & \vdots  & \vdots & \vdots & \vdots & {} & {}  \\
{} & {}  &  {} & {}  &  {}  &  0   & 0  &  0  & 0  & \hdots & {} \\
0 & {}  & \hdots & {}  & 0 &  b_n   & 0 &  0 & 0 & \hdots   \\[5pt] 
\hdashline 
0  & {} & \hdots & {} & 0 & 0 & b_{n+1}  & 0 & 0 & \hdots & {} \\
\vdots & {} & {}  & {} & \vdots & 0 & 0 & b_{n+2} & 0 & \hdots & {}   \\
{} & {} & {} & {}  & {} & 0 & 0& 0 & b_{n+3}  & {} & {}  \\
\vdots & {} & {} & {} & \vdots  & \vdots & \vdots & \vdots & \ddots & \ddots & {}  \\
\end{array}
\end{pmatrix}.$$
Now, since 
$$\norm{w-w_n}_{1,\infty} = \sup_{j\in \N}  \sum_{i\in \N}
\abs{w(i,j)-w_n(i,j)} ,$$
a simple calculation shows that 
$\norm{w-w_n}_{1,\infty} = 2\,\max_{i \geq n} b_i \leq 2 \,C \rho^n,$
which converges to $0$.

\bigskip

\item For all $n \geq 2$, $S$ is a structural set of type $B$ for $w_n$ (in the sense of Definition~\ref{structural:B}).
Fix $n \geq 2$. Let the function $M=M_{n} \colon \N \setminus S \to [0,+\infty)$ be defined by 
$M(i)=b_i$
and
consider the function $n_S=n_{S,w_n} \colon \N \setminus S \to \N$ introduced in Definition~\ref{structural:B}(3). 

We have that $n_S(i)=n-i+1$ for $i \in \{3,\ldots,n\}$, and $n_S(i)=1$ for 
$i \geq n+1$ (see Figure~\ref{markov:AB:finite:app}). 
Thus, taking in mind that the transition probabilities $(b_i)_{i \in \N}$ satisfy condition (B2), we have that
$$\sum_{i \in \N \setminus S} n_S(i)\,M(i) = \sum_{i=3}^{n}  (n-i+1) \, b_i + \sum_{i=n+1}^{\infty} b_i< +\infty.$$
We are left to check $(2)$ in Definition~\ref{structural:B}. We just need to  observe that, for all $i \in \{3,\ldots,n-1\}$,
$$ \abs{w_n(i,i+1)} = b_i \leq M(i) \;.$$

\end{enumerate}

\bigskip

We also have that:

\begin{enumerate}
\item \,\,$\lambda=1$ is an eigenvalue of  $\Aop_w$ and $\Aop_{w_n}$ for all $n \geq 2$.

\smallskip

\item \,\,  Since $1$ is an eigenvalue of $\Aop$ 
we can define a reduction operator 
$\Rop_{S,w}(1) \colon L^1(S) \to L^1(S)$
which keeps $1$ as an eigenvalue.
A simple calculation\footnote{We have used here that if $w=(w(i,j))_{i,j\in\N}$ is a  stochastic matrix such that $w(j,j)=0$ for all $j\notin S$  then $\Rop_{S,w}(1)$ is also a stochastic matrix. } shows that the $2 \times 2$ reduced matrix 
$\Rop_{S,w}(1)$ is given by

$$\Rop_{S,w}(1)=\left[
\begin{array}{lc}
1 - \sum_{\ell=0}^{\infty} \left(\prod_{k=1}^{\ell} b_{k+1} \right) a_{\ell +2} &  \hspace{0.2cm} b_1   \\[4pt] 

\sum_{\ell=0}^{\infty} \left(\prod_{k=1}^{\ell} b_{k+1} \right) a_{\ell +2} &  \hspace{0.2cm} 1-b_1    \\[4pt] 
\end{array}
\right].$$

\medskip
Let $v_0$ be an associated eigenvector.

\smallskip

\item \,\, Given  $n\geq 2$, the reconstruction operator 
 $\Phi_{S,w_ n} = \Phi_{S,w_ n}(1) \colon L^1(S) \to L^1(\N)$ can be characterized by  $ u_n=\Phi_{S,w_n}(v_0)$ with 
 $v_0=(v(1),v(2))$, 
$u_n=(u_n(i))_{i\in\N}$  and

$$ \left\{ \begin{array}{rclll}
u_n(i) &=&   v(i)  & \text{ if } & i \in S=\{1,2\}\\
\\
u_n(i) &=&  w_n(i,1) \,v(1)  +  \overbrace{w_n(i,2)}^{=\, 0}\,v(2) = a_i \,v(1)
  & \text{ if } & i \geq  n\\
  \\
  u_n(i) &=&  w_n(i,i+1)\,u_n(i+1)+
w_n(i,1)\,v(1) +  \overbrace{w_n(i,2)}^{=\, 0}\,v(2)
\\
{} & =&
\sum_{k=1}^{n-i+1}\left(\prod_{\ell=0}^{k-2}b_{i+\ell}\right) a_{i+k-1} \, v(1) 
  & \text{ if } & i \in \{3,\ldots,n-1\}\\
\end{array}\right.  . $$

\end{enumerate}

\bigskip

By Theorem~\ref{main:AB:countable}, the vector
$q_n=\frac{u_{n}}{\norm{u_{n}}_1}$ converges to the stationay probability measure $q$ defined in the statement of the proposition.

\section*{Acknowledgements}
PD was supported by ``Funda\c{c}\~{a}o para a Ci\^{e}ncia e a Tecnologia''
 through the Project UID/MAT/04561/2013.
 
MJT was partially supported by the Research Centre of Mathematics of the University of Minho with the Portuguese Funds from the ``Funda\c c\~ao para a Ci\^encia e a Tecnologia", through the Project UID/MAT/00013/2013.


\end{document}